\documentclass[letterpaper]{amsart}
\usepackage{tikz-cd}
\usepackage{amsmath,amssymb,amsthm}
\usepackage{mathrsfs}

\usepackage[utf8]{inputenc}
\usepackage{hyperref}
\theoremstyle{plain} 
\newtheorem{theorem}{Theorem}
\newtheorem*{theorem*}{Theorem}
\newtheorem{prop}[theorem]{Proposition}
\newtheorem{lemma}[theorem]{Lemma}
\newtheorem{coro}[theorem]{Corollary}

\theoremstyle{definition} \newtheorem{definition}{Definition}\theoremstyle{remark} \newtheorem*{remark*}{Remark}
\theoremstyle{definition} \newtheorem*{ack}{Acknowledgment}
\author{Chuanhao Wei}
\address{Department of Mathematics, University of Utah}
\email{wei@math.utah.edu}
\title[Log-Comparison with smooth boundary divisor in MHM]{Logarithmic Comparison with smooth boundary divisor in mixed Hodge modules}
\date{}
\begin{document}
\maketitle

\section{Introduction}
The main goal of this paper is to study certain filtered log-$\mathscr{D}$-modules that underlie the (dual) localization of Saito's mixed Hodge modules along a smooth hypersurface, (or more generally, it admits a multi-indexed Kashiwara-Malgrange filtration with respect to a normal crossing divisor as defined in Section 4,) and show that they also behave well under the direct image functor and the duality functor in the derived category of filtered log-$\mathscr{D}$-modules. We will apply the results of this paper to prove a natural and substantial generalization of the result of M. Popa and C. Schnell \cite{PS14} in the log setting. This generalization will appear in \cite{W17b}. Another application for the results in this paper is to simplify the proof of Viehweg's hyperbolicity for families of smooth varieties of general type in \cite{PS17}, and to generalize the result to the case of log-smooth families. Some other potential applications can also be expected in studying birational geometry of families of log-pairs over a log-smooth variety, e.g. subadditivity of log-Kodaira dimensions, \cite{Fuj17}.

We will consider right (log-)$\mathscr{D}$-modules in this paper, if not otherwise specified. The mixed Hodge modules that we are discussing here are all assumed to be algebraic. In particular, they are extendable \cite[\S 4]{Sa90}. We mainly follow the notations that appeared in \cite{SS16}. In particular, $\widetilde{\mathscr{D}}$ is the Rees algebra induced by the filtration $F$ on $\mathscr{D}$ given by the degree of the differential operators. We denote $\widetilde{\mathbb{C}}_X$ and $\widetilde{\mathcal{O}}_X$ the corresponding Rees algebra induced by the trivial filtration. See Section 2 for more details. We say a strict right $\widetilde{\mathscr{D}}$-module is a mixed Hodge module if it underlies a mixed Hodge module in the sense of Saito, \cite{Sa88} \cite{Sa90}, forgetting the weight filtration. All algebraic varieties that we work with in this paper are smooth and over the complex number field $\mathbb{C}$. 

Fix a normal crossing divisor $D$ on $X$. Let $\mathscr{D}_{(X,D)}$ be the ring of log-differential operators of the log-smooth pair $(X,D)$, which is canonically a sub-ring of $\mathscr{D}_X$. Similarly, denote $\widetilde{\mathscr{D}}_{(X,D)}$ the Rees algebra, induced by the filtration $F$ given by the degree of the differential operators.
The two $\widetilde{\mathscr{D}}_{(X,D)}$-modules that we are most interested in in this paper are the following:
\begin{definition}\label{D:1}
Let $\mathcal{M}$ be a right $\widetilde{\mathscr{D}}_X$-module that underlies a mixed Hodge module on a smooth variety $X$. Assume that $\mathcal{M}$ admits a V-compatible multi-indexed Kashiwara-Malgrange filtration with respect to the normal crossing divisor $D$, as in Definition \ref{D:V strict}. (In particular, it is the case if $D$ is a smooth divisor; or if $\mathcal{M}$ is of normal crossing type, with respect to $D'$, a normal crossing divisor that contains $D$.)
Then we denote:
\begin{align*}
\mathcal{M}[* \log D]&:=\mathbf{V}^{D}_{\mathbf{0}}\left(\mathcal{M}[*D]\right),\\
\mathcal{M}[! \log D]&:=\mathbf{V}^{D}_{<\mathbf{0}}\left(\mathcal{M}[!D]\right).
\end{align*}
See (\ref{E: def v fil}) and (\ref{E: def v fil'}) in Section 4, for the definition of the filtration $\mathbf{V}^{D}$. Note that, we have $\mathbf{V}^{D}_{<\mathbf{0}}\left(\mathcal{M}[!D]\right)=\mathbf{V}^{D}_{<\mathbf{0}}\left(\mathcal{M}[*D]\right)$. We pick the left one in the definition just for symmetry.

\end{definition}

We first generalize the logarithmic comparison theorem of Grothendieck-Deligne into the language of mixed Hodge modules. 
\begin{theorem}\label{T:Comparison Theorem Sp}
Under the assumptions in Definition \ref{D:1}, we have the following canonical quasi-isomorphisms
\begin{align*}
\textsl{Sp}_{\left(X,D\right)}\left(\mathcal{M}[* \log D]\right)&\simeq \textsl{Sp}_X\left(\mathcal{M}[*D]\right),\\
\textsl{Sp}_{\left(X,D\right)}\left(\mathcal{M}[! \log D]\right)&\simeq \textsl{Sp}_X\left(\mathcal{M}[!D]\right),
\end{align*}
in $DG\left(\widetilde{\mathbb{C}}_X\right)$, the derived category of graded $\widetilde{\mathbb{C}}_X$-modules. (See (\ref{E: log Sp}) in Section 2 for the definition of the $Sp$ functor.)
\end{theorem}

One special case of this theorem is that when $\mathcal{M}\big|_{X\setminus D}$ is an admissible variation of mixed Hodge structures, after a right-to-left side changing functor, replacing the Spencer functor by the de Rham functor, and forgetting the filtration, we recover the classical logarithmic comparison theorem: 
\begin{align*}
\textsl{DR}_{\left(X,D\right)}\left(\mathcal{V}^{*}\right)\simeq \mathbf{R}i_*\textsl{DR}_{X}\left(\mathcal{V}\right), (\text{resp. } \mathbf{R}i_!\textsl{DR}_{X}\left(\mathcal{V}\right),)
\end{align*}
where  $i:X\setminus D\to X$ is the natural embedding, $\left(\mathcal{V}^{*},\nabla^*\right)$ is Deligne's canonical extension that admits a logarithmic connection with
$$\text{Re}\{\text{eigenvalues of}\ \text{Res}_{D_i}\nabla^*\}\subset \left[0,1\right),$$
$$(\text{resp. } \text{Re}\{\text{eigenvalues of}\  \text{Res}_{D_i}\nabla^*\}\subset \left(0,1\right],)$$
and 
\begin{align*}
\textsl{DR}_{\left(X,D\right)}\left(\mathcal{V}^{*}\right):=[0\to \mathcal{V}^{*}\to \mathcal{V}^{*}\otimes \Omega^1_X &\left(\text{log}\ D\right)\to... \\ 
... &\to\mathcal{V}^{*}\otimes \Omega^{d_X}_X\left(\text{log}\ D\right)\to 0]
\end{align*}
is the log-de Rham complex of $\mathcal{V}^{*}$.
We refer to \cite{Wu16} for a survey on this topic. See also \cite[\S 3.]{Sa90}, for an exposition of the extensions in terms of mixed Hodge modules of normal crossing type.

Let $D=D_1+D_2+...+D_n$ be the decomposition of irreducible components. Given a subset $S\subset \{1,...,n\}$, we denote $D_S=\sum_{i\in S}D_i$.
The previous theorem can be immediately deduced by the following theorem, which is about the vanishing of the higher derived tensor operator:
\begin{theorem}\label{T:Comparison Theorem}
Given two subset $S\subset I\subset \{1,...,n\}$ and a mixed Hodge module $\mathcal{M}$, admitting a V-compatible multi-indexed Kashiwara-Malgrange filtration with respect to $D$, as in Definition \ref{D:V strict}, we have the following two quasi-isomorphisms in $DG_{\text{coh}}\left(\widetilde{\mathscr{D}}_{\left(X,D_X\right)}\right)$:
\begin{align}
\mathcal{M}[*D][* \log D_I]\otimes^{\mathbf{L}}_{\widetilde{\mathscr{D}}_{\left(X,D_I\right)}} \widetilde{\mathscr{D}}_{\left(X,D_S\right)} &\simeq  \mathcal{M}[*D][* \log D_S].\label{E:* comparison}\\
 \mathcal{M}[!D][! \log D_I]\otimes^{\mathbf{L}}_{\widetilde{\mathscr{D}}_{\left(X,D_I\right)}} \widetilde{\mathscr{D}}_{\left(X,D_S\right)} &\simeq   \mathcal{M}[!D][! \log D_S].\label{E:! comparison}
\end{align}
\end{theorem}

Note that, given any divisor $D$, the non-smooth locus of $D$ has codimension $\geq 2$ in $X$, hence the two isomorphisms always hold for any mixed Hodge module on a big open set of $X$. One big advantage of thinking of mixed Hodge modules is due to the fact that they behave well under the direct image functor of projective morphsims. We say that $f:\left(X,D^X\right)\to \left(Y, D^Y\right)$ is a projective morphism of log-smooth pairs if $X\to Y$ is a projective morphism and $f^{-1}D^Y\subset D^X$. The log-$\widetilde{\mathscr{D}}$-modules we considered above also behave well under direct image functors in the following sense:

\begin{theorem}\label{T:direct image}
Fix a projective morphism of log-smooth pairs $f:\left(X,D^X\right)\to \left(Y, D^Y\right)$, and a mixed Hodge module $\mathcal{M}$ admitting a V-compatible multi-indexed Kashiwara-Malgrange filtration with respect to $D^X$. Assume $D^Y$ is smooth. Then we have
\begin{align*}
\mathcal{H}^i f_\# \left( \mathcal{M}\left[*\log D^X\right]\right)&= \left(\mathcal{H}^i f_+\mathcal{M}\right)\left[* \log D^Y\right].\\
\mathcal{H}^if_\# \left(\mathcal{M}\left[!\log D^X\right]\right)&= \left(\mathcal{H}^i f_+\mathcal{M}\right)\left[! \log D^Y\right].
\end{align*}
In particular, the functor $f_\#$ is strict on $\mathcal{M}\left[*\log D^X\right]$ and $\mathcal{M}\left[!\log D^X\right]$.
\end{theorem}

In the case that $D^Y$ is only normal crossing instead of smooth, we need very strong extra conditions on the complex $f_+\mathcal{M}$, and see the remarks after the proof of the theorem in Section 6. However, in some birational geometry applications, we may be only interested in the behavior of the mixed Hodge module on $Y$ out side of a codimension $2$ subset. Hence we have $D^Y$ being smooth, by ignoring all singular locus of $D^Y$. In particular, in \cite{W17b}, we apply this theorem to $\widetilde{\omega}_X[!\log D^X]=\widetilde{\omega}_X$, where the first $\widetilde{\omega}_X$ is the $\widetilde{\mathscr{D}}$-module underlies the constant Hodge module, instead of considering the $\widetilde{\mathscr{D}}$-module direct image of $\widetilde{\omega}_X$ in \cite{PS14}. Even in the non-log setting, this simplifies the proof.

We also show that such log-$\widetilde{\mathscr{D}}$-modules behave well under the duality functor:
\begin{theorem}\label{T:dual}
Let $\mathcal{M}$ underlie a mixed Hodge module on $X$, admitting a V-compatible multi-indexed Kashiwara-Malgrange filtration with respect to a normal crossing divisor $D$. Let $\mathcal{M}'$ underlie the dual mixed Hodge module of $\mathcal{M}$:
$$\mathcal{M}':=\mathbf{R}\mathcal{H}om_{\widetilde{\mathscr{D}}_X}\left(\mathcal{M}, \widetilde{\omega}[d_X]\otimes \widetilde{\mathscr{D}}_X \right).$$
Then $\mathcal{M}'$ also admits a V-compatible multi-indexed Kashiwara-Malgrange filtration, with respect to $D$, and further
\begin{align*}
\mathbf{D}_{(X,D)} \left(\mathcal{M}[*\log D]\right)&\simeq \mathcal{M}'[!\log D]\\
\mathbf{D}_{(X,D)} \left(\mathcal{M}[!\log D]\right)&\simeq \mathcal{M}'[*\log D]
\end{align*}
(See (\ref{E:dual definition}) in Section 2 for the definition of the duality functor $\mathbf{D}_{(X,D)}$.)
In particular, we have that both $\mathbf{D}_{(X,D)} \left(\mathcal{M}[*\log D]\right)$ and $\mathbf{D}_{(X,D)} \left(\mathcal{M}[!\log D]\right)$ are strict.
\end{theorem}

The existence of the Hodge filtration and its compatibility with the Kashiwara-Malgrange filtration is essential to prove the theorems. By keeping track of the Hodge filtration and the strictness properties that we obtain, we can also get some explicit formulae for the associated graded modules (\cite[2.1]{Sa88}): Corollary \ref{C:direct image gr}, Corollary \ref{C:dual gr}, which can be very useful in some geometric applications, as in \cite{W17b}.

We first fix the notations in Section 2, 3, 4. They will be stated in the general setting of filtered (log-)$\mathscr{D}$-modules. In Section 5, we prove the comparison theorem. The theorems about the direct image functor and duality functor will be proved in Section 6 and 7, respectively.

\begin{ack}
The author would like to express his gratitude to his adviser Christopher Hacon for suggesting this topic and useful discussions. The author thanks Linquan Ma, Christian Schnell, Lei Wu and Ziwen Zhu for answering his questions, and an anonymous referee that suggests a method, which is showing here, that simplifies the proof of Lemma \ref{L:vanishing by torsion freeness}.

During the preparation of this paper, the author was partially supported by DMS-1300750 and a grant from the Simons Foundation, Award Number 256202.
\end{ack}

\section{Preliminaries}
In this section, we first define the log-$\mathscr{D}$-module corresponding to a log-smooth pair. Then we recall the notions appeared in \cite[Appendix A.]{SS16}, but in the log $\widetilde{\mathscr{D}}$-modules setting. We will also recall some basic functors: the Spencer functor, the pushforward functor and the duality functor, on the derived category of log-$\widetilde{\mathscr{D}}$ modules. If we set the boundary divisor $D=0$, we will get the same functors on the derived category of $\widetilde{\mathscr{D}}$ modules. 

Given a log-smooth pair $(X,D)$, i.e., a smooth variety $X$ with a  normal crossing divisor $D$ on $X$, with $D=D_1+...+D_n$ being its decomposition of irreducible components. Denote $\dim X=d_X$. Working locally on $X$, we can assume that $X$ is just a polydisc $\Delta^{d_X}$ in $\mathbb{C}^{d_X}$. Let $x_1,...,x_{d_X}$ be a local analytic coordinate system on $X$ and $D$ is the simply normal crossing (SNC) divisor that defined by $y:=x_1\cdot...\cdot x_n=0$. Let $\partial_1,...,\partial_n$ be the dual basis of $dx_1,...,dx_{d_X}$. We define $\mathcal{T}_{\left(X,D\right)}$, the sheaf of the log-tangent bundle on $\left(X,D\right)$, to be the locally free sheaf that is locally generated over $\mathscr{O}_X$ by $\{x_1\partial_1,...,x_n\partial_n,\partial_{n+1},...,\partial_{d_X}\}$. $\mathcal{T}_{\left(X,D\right)}$ is naturally a sub-sheaf of $\mathcal{T}_X$, the sheaf of the tangent bundle over $X$. Similarly, we define $\mathscr{D}_{\left(X,D\right)}$, the sheaf of logarithmic differential operators on $\left(X,D\right)$, to be the sub-$\mathscr{O}_X$-algebra of $\mathscr{D}_X$ that is locally generated by $\{x_1\partial_1,...,x_n\partial_n,\partial_{n+1},...,\partial_{d_X}\}$. If $D=0$, $\mathscr{D}_{\left(X,D\right)}=\mathscr{D}_X$.

There is a natural filtration $F$ on $\mathscr{D}_X$ given by the degree of the differential operators. 
We denote 
$$\widetilde{\mathscr{D}}_{X}=R_F \mathscr{D}_{X}:=\oplus_p F_p\mathscr{D}_{X} ,$$
 the Rees algebra induced by $\left(\mathscr{D}_X,F\right).$ 
For any differential operator $\partial_t\in F_1\mathscr{D}_{X}$, we denote 
$$\widetilde{\partial}_t:=\partial_t\cdot z\in R_F \widetilde{\mathscr{D}}_{X}.$$ 

Denote by
$MG\left(\widetilde{\mathscr{D}}_{X}\right)$, the category of graded $\widetilde{\mathscr{D}}_{X}$ modules whose morphisms are graded morphisms of degree zero, 
and we call it the category of associated Rees modules of the category of filtered $\mathscr{D}_{X}$-modules $MF\left(\mathscr{D}_{X}\right)$.  Note that $MG\left(\widetilde{\mathscr{D}}_{X}\right)$ is an abelian category. See \cite[Definition A.2.3]{SS16} for details.
Hence, we can define the derived category of $MG\left(\widetilde{\mathscr{D}}_{X}\right)$, and we denote it by $DG\left(\widetilde{\mathscr{D}}_{X}\right)$. We also use $D^*G\left(\widetilde{\mathscr{D}}_{X}\right), *\in\{-, +, b\}$ to denote the bounded above, bounded below, bounded condition respectively on the derived category.

Recall that, we say $\mathcal{M}\in MG\left(\widetilde{\mathscr{D}}_{X}\right)$ is a coherent $\widetilde{\mathscr{D}}_{X}$-module if $\mathcal{M}$ is locally finitely presented, i.e. if for any $x\in X$, there
exists an open neighborhood $U_x$ of $x$ and an exact sequence:
$$\widetilde{\mathscr{D}}_{X}^q\big|_{U_x}\to \widetilde{\mathscr{D}}_{X}^p\big|_{U_x}\to \mathcal{M}\big|_{U_x}.$$
We denote by $MG_{\text{coh}}\left(\widetilde{\mathscr{D}}_{X}\right)$, the full subcategory $MG\left(\widetilde{\mathscr{D}}_{X}\right)$ with the objects that are coherent. Note that we know $MG_{\text{coh}}\left(\widetilde{\mathscr{D}}_{X}\right)$ is also an abelian category. We refer to \cite[A.9. A.10.]{SS16} for a discussion on this topic.

Fix $\mathcal{M}\in MG\left(\widetilde{\mathscr{D}}_{X}\right)$, we can write $\mathcal{M}_p$, the $p$-th graded piece as 
$$\mathcal{M}_p=F_p\mathcal{M}\cdot z^p,$$
where $F_p\mathcal{M}$ is an $\mathscr{O}_X$-module. If $\mathcal{M}$ is a coherent $\widetilde{\mathscr{D}}_{X}$-module, then $F_p\mathcal{M}$ is a coherent $\mathscr{O}_X$-module.

Similarly, we have a natural filtration $F$ on $\mathscr{D}_{\left(X,D\right)}$, and denote 
$$\widetilde{\mathscr{D}}_{\left(X,D\right)}=R_F \mathscr{D}_{\left(X,D\right)}:=\oplus_p F_p\mathscr{D}_{\left(X,D\right)} .$$
We similarly use $MG_{(\text{coh})}\left(\widetilde{\mathscr{D}}_{\left(X,D\right)}\right)$, $D^*G\left(\widetilde{\mathscr{D}}_{\left(X,D\right)}\right)$ to denote the corresponding categories as above.

Let $\widetilde{\mathscr{O}}_X$ be the sheaf given by the graded ring $\mathscr{O}_X[z]=R_F \mathscr{O}_X$, where the filtration $F$ on $\mathscr{O}_X$ is the trivial one. We denote $\widetilde{\mathbb{C}}_X=\mathbb{C}_X[z]$ as a graded sub-ring of $\widetilde{\mathscr{O}}_X$.
Further, given a coherent $\mathscr{O}_X$-module $\mathcal{L}$, we denote $\widetilde{\mathcal{L}}=\mathcal{L}[z]$, the induced graded 
$\widetilde{\mathscr{O}}_X$-module. In particular, we have $\widetilde{\omega}_X$, although it is also a $\widetilde{\mathscr{D}}_X$-module. From now on, we will omit the word ``graded'' as long as it is clear from the context.

As in \cite[A.5.]{SS16} and \cite[3.1]{CM99}, we define the logarithmic Spencer complex on a $\widetilde{\mathscr{D}}_{\left(X,D\right)}$-module $\mathcal{M}$, by
\begin{equation}\label{E: log Sp}
\textsl{Sp}_{\left(X,D\right)}\left(\mathcal{M}\right):=\{0\to \mathcal{M}\otimes\wedge^{n}\widetilde{\mathcal{T}}_{\left(X,D\right)}\to ...\to \mathcal{M}\otimes\widetilde{\mathcal{T}}_{\left(X,D\right)}\to \mathcal{M}\to 0\},
\end{equation}
with the $\widetilde{\mathbb{C}}_X$-linear differential map locally given by 
\begin{align*}
m\otimes \xi_1\wedge...\wedge\xi_k\mapsto \sum_{i=1}^k & \left(-1\right)^{i-1}m\xi_i\otimes \xi_1\wedge...\wedge\hat{\xi}_i\wedge...\wedge\xi_k +\\
&\sum_{1\leq i<j\leq k} (-1)^{i+j} m\otimes [\xi_i,\xi_j]\wedge \xi_1\wedge...\wedge\hat{\xi}_i\wedge...\wedge\hat{\xi}_j\wedge...\wedge\xi_k,
\end{align*}
where $\xi_1,...,\xi_{d_X}$ is a local basis of $\mathcal{T}_{\left(X,D\right)}\cdot z$. 
We have the following log version of \cite[Theorem 3.1.2.]{CM99}:
\begin{theorem}\label{T:resolution of O as D-module}
$\textsl{Sp}_{\left(X,D\right)}\left(\widetilde{\mathscr{D}}_{\left(X,D\right)}\right)$ is a resolution of $\widetilde{\mathscr{O}}_X$ as left $\widetilde{\mathscr{D}}_{\left(X,D\right)}$-modules.
\end{theorem}

Hence, we can view the logarithmic Spencer complex in the following way. We have the right exact functor:
$$\otimes_{\widetilde{\mathscr{D}}_{\left(X,D\right)}} \widetilde{\mathscr{O}}_X: MG\left(\widetilde{\mathscr{D}}_{\left(X,D\right)}\right)\to MG\left(\widetilde{\mathbb{C}}_X\right),$$
by
$$\mathcal{M}\mapsto \mathcal{M}\otimes_{\widetilde{\mathscr{D}}_{\left(X,D\right)}} \widetilde{\mathscr{O}}_X.$$
We can define its left derived functor
$$\textsl{Sp}_{\left(X,D\right)}=\otimes^{\mathbf{L}}_{\widetilde{\mathscr{D}}_{\left(X,D\right)}} \widetilde{\mathscr{O}}_X: D^*G\left(\widetilde{\mathscr{D}}_{\left(X,D\right)}\right)\to D^*G\left(\widetilde{\mathbb{C}}_X\right),$$
by
$$\textsl{Sp}_{\left(X,D\right)}\mathcal{M}^\bullet= \mathcal{M}^\bullet\otimes^{\mathbf{L}}_{\widetilde{\mathscr{D}}_{\left(X,D\right)}} \widetilde{\mathscr{O}}_X.$$
Note that each term of $\textsl{Sp}_{\left(X,D\right)}\left(\widetilde{\mathscr{D}}_{\left(X,D\right)}\right)$ is locally free over $\widetilde{\mathscr{D}}_{\left(X,D\right)}$. Hence due to Theorem \ref{T:resolution of O as D-module}, we have $\textsl{Sp}_{\left(X,D\right)}\left(\mathcal{M}\right)\simeq \mathcal{M}\otimes^{\mathbf{L}}_{\widetilde{\mathscr{D}}_{\left(X,D\right)}} \widetilde{\mathscr{O}}_X$

To relieve the burden of notation, in the rest of the paper, $\otimes$ means $\otimes_{\widetilde{\mathscr{O}}}$ or $\otimes_{\mathscr{O}}$, depending on the context, if not otherwise specified.

As in the $\widetilde{\mathscr{D}}_X$ case \cite[2.4.1. 2.4.2.]{Sa88} \cite[Exercise A.19]{SS16}, there are two right $\widetilde{\mathscr{D}}_{\left(X,D\right)}$-module structures on $\mathcal{M}\otimes \widetilde{\mathscr{D}}_{\left(X,D\right)}$, which are defined by:
\begin{align*}
\text{(right)}_\text{triv}:&\ 
\begin{cases}
\left(m\otimes P\right)\cdot_\text{triv} f=m\otimes \left(Pf\right)\\
\left(m\otimes P\right)\cdot_\text{triv} \xi=m\otimes \left(P\xi\right)
\end{cases}
\\
\text{(right)}_\text{tens}:&\ 
\begin{cases}
\left(m\otimes P\right)\cdot_\text{tens} f=mf\otimes P\\
\left(m\otimes P\right)\cdot_\text{tens} \xi=m\xi\otimes P-m\otimes \left(\xi P\right),
\end{cases}
\end{align*}
where $f\in \mathcal{O}_X$, and $\xi\in F_1\mathscr{D}_{\left(X,D\right)}\cdot z$. We call the first one the trivial right $\widetilde{\mathscr{D}}_{\left(X,D\right)}$-module structure, and the second one the right $\widetilde{\mathscr{D}}_{\left(X,D\right)}$-module structure induced by tensor product.
\begin{prop}\label{P:involution}
Notations as above, there is a unique involution $\tau: \mathcal{M}\otimes \widetilde{\mathscr{D}}_{\left(X,D\right)}\to \mathcal{M}\otimes \widetilde{\mathscr{D}}_{\left(X,D\right)}$ which exchanges both structures and is the identity on $\mathcal{M}\otimes 1$. It is given by 
$$\left(m\otimes P\right)\mapsto (m\otimes 1)\cdot_\text{tens}P.$$
\end{prop}
Since it can be checked directly, we omit the proof here. See \cite[2.4.2.]{Sa88} for the details.

Recall that we have a natural resolution for any $\widetilde{\mathscr{D}}_{\left(X,D\right)}$-module.

\begin{align}\label{E:M otimes spencer}
\mathcal{M}\simeq &\mathcal{M}\otimes \textsl{Sp}_{\left(X,D\right)}\left(\widetilde{\mathscr{D}}_{\left(X,D\right)}\right)\\
\simeq &\{0\to \left(\mathcal{M}\otimes\wedge^{n}\widetilde{\mathcal{T}}_{\left(X,D\right)}\otimes \widetilde{\mathscr{D}}_{\left(X,D\right)}\right)_\text{triv}\to ...\nonumber\\
&\to \left(\mathcal{M}\otimes\widetilde{\mathcal{T}}_{\left(X,D\right)}\otimes \widetilde{\mathscr{D}}_{\left(X,D\right)}\right)_\text{triv}\to \left(\mathcal{M}\otimes \widetilde{\mathscr{D}}_{\left(X,D\right)}\right)_\text{triv}\to 0\}.\nonumber
\end{align}
It has been built in \cite[2.1.6. Lemme, 2.2.8. Lemme.]{Sa88} in the filtered $\mathscr{D}$-module case.

To define the pushforward functor, we denote 
$$\widetilde{\mathscr{D}}_{\left(X,D^X\right)\to \left(Y,D^Y\right)}=\widetilde{\mathscr{O}}_X\otimes_{f^{-1}\widetilde{\mathscr{O}}_Y}f^{-1}\widetilde{\mathscr{D}}_{\left(Y,D^Y\right)},$$ 
which is a left $\widetilde{\mathscr{D}}_{\left(X,D^X\right)}$, right $f^{-1}\widetilde{\mathscr{D}}_{\left(Y,D^Y\right)}$-module.

Following a similar route in \cite[A.8.a]{SS16}, we define the pushforward functor, 
$$f_\#: D^+G\left(\widetilde{\mathscr{D}}_{\left(X,D^X\right)}\right)\to D^+G\left(\widetilde{\mathscr{D}}_{\left(Y,D^Y\right)}\right),$$
by
\begin{equation}\label{E:direct image}
f_\# \mathcal{M}^\bullet=\mathbf{R}f_*\left( \mathcal{M}^\bullet \otimes^{\mathbf{L}}_ {\widetilde{\mathscr{D}}_{\left(X,D^X\right)}}\widetilde{\mathscr{D}}_{\left(X,D^X\right)\to \left(Y,D^Y\right)} \right).
\end{equation}

If $\mathcal{M}$ is an object in $MG_{(\text{coh})}\left(\widetilde{\mathscr{D}}_{\left(X,D\right)}\right)$, then each cohomology of $f_\# \mathcal{M}$ is an object in $MG_{(\text{coh})}\left(\widetilde{\mathscr{D}}_{\left(X,D\right)}\right)$, \cite[Theorem A.10.15]{SS16}.

Note that if both $D^X$ and $D^Y$ are trivial, we have $f_\#=f_+$, where $f_+$ is the pushforward functor on the derived category of $\widetilde{\mathscr{D}}$-modules. (In \cite{SS16}, they use the notation $_Df_*$ instead of $f_+$.)

Further, in the situation of Theorem \ref{T:Comparison Theorem}, considering the identity map $id:(X,D_I)\to (X,D_S)$, we have
\begin{align*}
id_\#\mathcal{M}[*D][* \log D_I]&=\mathcal{M}[*D][* \log D_I]\otimes^{\mathbf{L}}_{\widetilde{\mathscr{D}}_{\left(X,D_I\right)}} \widetilde{\mathscr{D}}_{\left(X,D_S\right)}
\\
id_\#\mathcal{M}[!D][! \log D_I]&=\mathcal{M}[!D][! \log D_I]\otimes^{\mathbf{L}}_{\widetilde{\mathscr{D}}_{\left(X,D_I\right)}} \widetilde{\mathscr{D}}_{\left(X,D_S\right)}.
\end{align*}

\begin{prop}
Let $f:\left(X, D^X\right)\to \left(Y, D^Y\right)$ and $g:\left(Y,D^Y\right)\to \left(Z,D^Z\right)$ be two morphisms of log-smooth pairs, and  assume that $f$ is proper, we have a functorial canonical isomorphism of functors
$$\left(g\circ f\right)_\#\simeq g_\#\circ f_\#.$$
\end{prop}
It is just a log version of \cite[Theorem A.8.16. Remark A.8.17.]{SS16}, so we omit the proof here.

Recall that we define the duality functor, \cite[2.4.3]{Sa88},
$$\mathbf{D}_{X}:D^-G\left(\widetilde{\mathscr{D}}_{X}\right)\to D^+G\left(\widetilde{\mathscr{D}}_{X}\right),$$
by
$$
\mathbf{D}_{X}\mathcal{M}^\bullet=\mathbf{R}\mathcal{H}om_{\widetilde{\mathscr{D}}_{X}}\left(\mathcal{M}^\bullet, \left(\widetilde{\omega}_X[d_X]\otimes \widetilde{\mathscr{D}}_{X}\right)_{\text{tens}}\right),
$$
where $\widetilde{\omega}_X[d_X]$ means shifting the sheaf $\widetilde{\omega}_X$ to the cohomology degree $-d_X$.

Similarly, we define the duality functor
$$\mathbf{D}_{\left(X,D^X\right)}:D^-G\left(\widetilde{\mathscr{D}}_{\left(X,D^X\right)}\right)\to D^+G\left(\widetilde{\mathscr{D}}_{\left(X,D^X\right)}\right),$$
by
\begin{equation}\label{E:dual definition}
\mathbf{D}_{\left(X,D^X\right)}\mathcal{M}^\bullet=\mathbf{R}\mathcal{H}om_{\widetilde{\mathscr{D}}_{\left(X,D^X\right)}}\left(\mathcal{M}^\bullet, \left(\widetilde{\omega}_X[d_X]\otimes \widetilde{\mathscr{D}}_{\left(X,D^X\right)}\right)_{\text{tens}}\right),
\end{equation}

We use the induced tensor right $\widetilde{\mathscr{D}}_{\left(X,D^X\right)}$-module structure on $\widetilde{\omega}_X\otimes \widetilde{\mathscr{D}}_{\left(X,D^X\right)}$ when we take the $\mathbf{R}\mathcal{H}om_{\widetilde{\mathscr{D}}_{\left(X,D^X\right)}}\left(\mathcal{M}^\bullet, -\right)$ functor, and we use the trivial right $\widetilde{\mathscr{D}}_{\left(X,D^X\right)}$-module structure on $\widetilde{\omega}_X\otimes \widetilde{\mathscr{D}}_{\left(X,D^X\right)}$ to give the right $\widetilde{\mathscr{D}}_{\left(X,D^X\right)}$-module structure on $\mathbf{D}_{\left(X,D^X\right)}\mathcal{M}^\bullet$.

\section{Strictness Condition} 
In this section, we recall the definition of the strictness condition. We also recall the associated graded functor in this section. At the end of the section, we explicitly write down the formulae for the induced direct image functor and duality functor on the associated graded module, under certain strictness condition. They can be very useful in some geometric applications.   

Fix a filtered ring $\left(\mathcal{R},F\right)$. We always assume the filtration $F$ is exhaustive, which means $\cup_i F_i\mathcal{R}=\mathcal{R}$, and $F_{\ll 0}=0$. Denote $\widetilde{\mathcal{R}}:=R_F\mathcal{R}$, the associated Rees algebra. We define the strictness condition on $MG(\widetilde{\mathcal{R}})$ by \cite[Definition A.2.5.]{SS16}:
\begin{definition}\ \\
(1) An object of $MG(\widetilde{\mathcal{R}})$ is said to be strict, if it has no $\mathbb{C}[z]$-torsion, i.e., comes
from a filtered $\mathcal{R}$-module.\\
(2) A morphism in $MG(\widetilde{\mathcal{R}})$ is said to be strict, if its kernel and cokernel are strict. (Note that the composition of two strict morphisms need not be strict).\\
(3) A complex $\mathcal{M}^{\bullet}$ of $MG(\widetilde{\mathcal{R}})$ is said to be strict, if each of its cohomology modules is a strict object of $MG(\widetilde{\mathcal{R}})$. (Hence we can use the same definition for the strictness of $\mathcal{M}^{\bullet}\in D^*G(\widetilde{\mathcal{R}})$.
\end{definition}
We have a right exact functor
$$\text{Gr}^F: MG(\widetilde{\mathcal{R}})\to MG\left(\text{Gr}^F\mathcal{R}\right),$$
defined by $\text{Gr}^F\left(\mathcal{M}\right)=\mathcal{M}\otimes_{\widetilde{\mathbb{C}}}(\widetilde{\mathbb{C}}/z\widetilde{\mathbb{C}}).$
It naturally induces a derived functor:
$$\widetilde{\text{Gr}}^F: D^*G(\widetilde{\mathcal{R}})\to D^*G\left(\text{Gr}^F\mathcal{R}\right)$$
given by 
$$\widetilde{\text{Gr}}^F \mathcal{M}^\bullet=\mathcal{M}^\bullet\otimes^{\mathbf{L}}_{\widetilde{\mathbb{C}}}\widetilde{\mathbb{C}}/z\widetilde{\mathbb{C}}.$$

In particular, if we assume that $\mathcal{M}\in MG(\widetilde{\mathcal{R}})$ is strict, which means $\mathcal{M}$ is the associated Rees module of a filtered $\mathcal{R}$-module, we have
$$\text{Gr}^F\mathcal{M}\simeq \widetilde{\text{Gr}}^F \mathcal{M}.$$

Now given $\mathcal{M}^{\bullet}\in DG(\widetilde{\mathcal{R}})$, we have a spectral sequence 
$$E_2^{p,q}=\mathcal{H}^{p}\widetilde{\text{Gr}}^F \mathcal{H}^q\mathcal{M}^\bullet\Rightarrow \mathcal{H}^{p+q} \widetilde{\text{Gr}}^F \mathcal{M}^\bullet.$$
If $\mathcal{M}^\bullet$ is strict, we have the spectral sequence degenerates at the $E_2$ page, which implies the commutativity of taking cohomology and $\widetilde{\text{Gr}}^F$ functor under the strictness condition:
$$\widetilde{\text{Gr}}^F \mathcal{H}^i \mathcal{M}^\bullet=\mathcal{H}^{i} \widetilde{\text{Gr}}^F \mathcal{M}^\bullet.$$

In the case that $\widetilde{\mathcal{R}}=\widetilde{\mathscr{D}}_{\left(X,D\right)}$, we have  
$$\mathcal{A}_{\left(X,D\right)}:=\widetilde{\text{Gr}}^F\widetilde{\mathscr{D}}_{\left(X,D\right)}= \textsl{Sym} \mathcal{T}_{\left(X,D\right)}.$$
We denote 
$$T^*_{\left(X,D\right)}=\text{Spec}\left(\mathcal{A}_{\left(X,D\right)}\right),$$
the space of log-cotangent bundle over $\left(X,D\right)$. Hence, we have a canonical functor
$$\widetilde{\text{Gr}}^F: D^*G\left(\widetilde{\mathscr{D}}_{\left(X,D\right)}\right)\to D^*G\left(\mathcal{A}_{\left(X,D\right)}\right),$$
We denote by $\mathcal{G}\left(\mathcal{M}^{\bullet }\right),$ the object in the derived category of $\mathscr{O}_{T^*_{\left(X,D\right)}}$-modules that is induced by $\widetilde{\text{Gr}}^F\mathcal{M}^\bullet$. 

Given a strict $\widetilde{\mathscr{D}}_{\left(X,D^X\right)}$-module $\mathcal{M}$, the following two propositions give explicit formulae for the direct image functor and duality functor induced on $\text{Gr}^F\mathcal{M}$. See \cite[Theorem 2.3, 2.4]{PS13} for similar statements for mixed Hodge modules. The statements here are more general and we give a simplified and functorial proof.
\begin{prop}[Laumon's formula]\label{P:laumon}
Let $f:\left(X,D^X\right)\to \left(Y,D^Y\right)$ be a morphism of log-smooth pairs.
Let $\mathcal{M}\in MG\left(\widetilde{\mathscr{D}}_{\left(X,D^X\right)}\right)$. Assume that both $\mathcal{M}$ and $f_\#\mathcal{M}^\bullet$ are strict. Then we have
$$\mathcal{H}^i f_{\widetilde{\#}}\text{Gr}^F\mathcal{M}:=R^i f_*\left(\text{Gr}^F\mathcal{M}\otimes^{\mathbf{L}}_{\mathcal{A}_{\left(X,D^X\right)}}f^*\mathcal{A}_{\left(Y,D^Y\right)}\right)\simeq\text{Gr}^F \mathcal{H}^i f_\# \mathcal{M}.$$
\end{prop}
\begin{proof}
Denote 
$$\mathcal{N}^\bullet=\mathcal{M}^\bullet \otimes^{\mathbf{L}}_ {\widetilde{\mathscr{D}}_{\left(X,D^X\right)}}\left(\widetilde{\mathscr{O}}_X\otimes_{f^{-1}\widetilde{\mathscr{O}}_Y}f^{-1}\widetilde{\mathscr{D}}_{\left(Y,D^Y\right)}\right) .$$
By the associativity of tensor product and $\mathcal{M}$ being strict, we have  
$$\widetilde{\text{Gr}}^F \mathcal{N}^\bullet=\text{Gr}^F\mathcal{M}\otimes^{\mathbf{L}}_{\mathcal{A}_{\left(X,D^X\right)}}f^*\mathcal{A}_{\left(Y,D^Y\right)}.$$
Further, by the projection formula, we have 
$$R^if_* \widetilde{\text{Gr}}^F \mathcal{N}\simeq \mathcal{H}^i\widetilde{\text{Gr}}^F\mathbf{R}f_*\mathcal{N}.$$
Now we only need to show that 
$$\mathcal{H}^i\widetilde{\text{Gr}}^F\mathbf{R}f_*\mathcal{N}\simeq \widetilde{\text{Gr}}^F R^if_*\mathcal{N}.$$
It follows evidently by the commutativity of taking cohomology and $\widetilde{\text{Gr}}^F$ functor under the strictness assumption on $f_\#\mathcal{M}^\bullet\simeq \mathbf{R}f_*\mathcal{N}$.
\end{proof}

\begin{prop}\label{P:dual laumon}
Fix a $\mathcal{M}\in MG\left(\widetilde{\mathscr{D}}_{\left(X,D\right)}\right)$ and assume $\mathbf{D}_{\left(X,D\right)}\mathcal{M}\in MG\left(\widetilde{\mathscr{D}}_{\left(X,D\right)}\right),$ which means it has only one non-trivial cohomology at cohomology degree $0$. We further assume that both $\mathcal{M}$ and $\mathbf{D}_{\left(X,D\right)}\mathcal{M}$ are strict.  Then we have 
$$\text{Gr}^F \mathbf{D}_{\left(X,D\right)}\mathcal{M}\simeq \mathbf{R}\mathcal{H}om_{\mathcal{A}_{\left(X,D\right)}}\left(\text{Gr}^F \mathcal{M}, \omega_X[d_X]\otimes_{\mathscr{O}_X}\mathcal{A}_{\left(X,D\right)}\right).$$
Note that the sections of $\mathcal{A}_{\left(X,D\right),k}:=\textsl{Sym}^k\mathcal{T}_{\left(X,D\right)}$ act on the right-hand with an extra factor of $(-1)^k$. It is due to that we are using the induced tenser log-$\mathscr{D}$ module structure on the right-hand side.
If we consider both sides as $\mathscr{O}_{T^*_{\left(X,D\right)}}$-complexes, we get
$$\mathcal{G}\left(\mathbf{D}_{\left(X,D\right)}\mathcal{M}\right)\simeq \left(-1\right)^*_{T^*_{\left(X,D\right)}}
\mathbf{R}\mathcal{H}om_{\mathscr{O}_{T^*_{\left(X,D\right)}}}
\left(\mathcal{G}(\mathcal{M}), p_X^*\omega_X[d_X]
\otimes \mathscr{O}_{T^*_{\left(X,D\right)}}\right).$$
\end{prop}
\begin{proof}
Since $\mathcal{M}$ is strict, we have
\begin{align*}
&\mathbf{R}\mathcal{H}om_{\widetilde{\mathscr{D}}_{\left(X,D\right)}}\left(\mathcal{M}, \left(\widetilde{\omega}_X[d_X]\otimes \widetilde{\mathscr{D}}_{\left(X,D\right)}\right)_{\text{tens}}\right)\otimes^{\mathbf{L}}_{\widetilde{\mathbb{C}}}\left(\widetilde{\mathbb{C}}/z\widetilde{\mathbb{C}}\right)\\
\simeq&  \mathbf{R}\mathcal{H}om_{\mathcal{A}_{\left(X,D\right)}}\left(\mathcal{M}\otimes_{\widetilde{\mathbb{C}}}\left(\widetilde{\mathbb{C}}/z\widetilde{\mathbb{C}}\right), \omega_X[d_X]\otimes_{\mathscr{O}_X}\mathcal{A}_{\left(X,D\right)}\right),
\end{align*}
which can be checked locally by taking a resolution of $
\mathcal{M}$ by free $\widetilde{\mathscr{D}}_{\left(X,D\right)}$-modules. Now the statement is clear by the strictness assumption on $\mathbf{D}_{\left(X,D\right)}\mathcal{M}$.
\end{proof}

\section{Kashiwara-Malgrange filtration on coherent $\widetilde{\mathscr{D}}_X$-modules}
In this section, we recall some basic properties of $\mathbb{R}$-indexed Kashiwara-Malgrange filtration on a coherent strict $\widetilde{\mathscr{D}}$-module with respect to a smooth divisor from \cite[3.1]{Sa88}. Then we define multi-indexed rational Kashiwara-Malgrange filtration with respect to a simply normal crossing divisor when the $\widetilde{\mathscr{D}}$-module is of normal crossing type, and recall some properties that we will need in later sections.

We say that an $\mathbb{R}$-indexed increasing filtration $V$ is indexed by $A+\mathbb{Z}$, where $A$ is a finite subset of $\left[-1,0\right)$, if $gr^V_a:=V_a/V_{a<0}=0$ if and only if $a\notin A+\mathbb{Z}.$ All $\mathbb{R}$-indexed filtrations in this paper are indexed by $A+\mathbb{Z}$ for some finite set $A\subset \left[-1,0\right)$. In this paper, all $\mathbb{R}$-indexed increasing filtrations are assumed to be indexed by $A+\mathbb{Z}$, for some finite set $A\subset \left[-1,0\right)$. 

Fix a smooth variety $X$ and a smooth divisor $H$ on $X$. Let $t$ be a local function that defines $H$, and $\partial_t$ be a local vector field satisfying $[\partial_t, t]=1$.
\begin{definition}\label{D:KM-filtration on H}
Let $\mathcal{M}$ be a coherent strict $\widetilde{\mathscr{D}}_X$-module. We say that an $\mathbb{R}$-indexed increasing filtration $V^H_\bullet$ on $\mathcal{M}$ is a \emph{Kashiwara-Malgrange filtration} with respect to $H$, if\\
(1)$\cup_{a\in \mathbb{R}}V^H_a \mathcal{M}=\mathcal{M}$, and each filtered piece $V_a \mathcal{M}$ is a coherent $\widetilde{\mathscr{D}}_{\left(X,H\right)}$-module.\\
(2)$\left(V_a^H \mathcal{M}\right)\left(V_i^H \widetilde{\mathscr{D}}_X\right)\subset V_{a+i}^H \mathcal{M}$ for all $a\in \mathbb{R}, i\in \mathbb{Z}$, and
$\left(V^H_a \mathcal{M}\right)t=V^H_{a-1}\mathcal{M}$ if $a<0$.\\
(3)The action $t\widetilde{\partial}_t-a$ over $\text{gr}^{V^H}_a \mathcal{M}$ is nilpotent for any $a\in \mathbb{R}$.\\
Note that $t\in V^H_{-1}\widetilde{\mathscr{D}}_X$, $\widetilde{\partial}_t=\partial_t\cdot z\in V^H_1\widetilde{\mathscr{D}}_X$, hence condition (3) implies\\
(4)$t: \text{gr}^{V^H}_a \mathcal{M}\to \text{gr}^{V^H}_{a-1} \mathcal{M}$ and $\widetilde{\partial}_t: \text{gr}^{V^H}_{a-1} \mathcal{M}\to  \text{gr}^{V^H}_a \mathcal{M}z$ are bijective for $a\neq 0$.\\
(5)$V^H_{a+i}\mathcal{M}=\left(V_a \mathcal{M}\right)\left(V_i^H \widetilde{\mathscr{D}}_X\right)$ for $a\geq0$, $i\geq 0$.\\
Further, all the conditions above are independent from the choice of $t$ and $\widetilde{\partial}_t$.
\end{definition}

We know that given a mixed Hodge module $\mathcal{M}$, in particular being strictly $\mathbb{R}$-specializable along $H$, there exists a rational Kashiwara-Malgrange filtration with respect to $H$ and it is unique, \cite[3.1.2. Lemme.]{Sa88}, \cite[7.3.c]{SS16}. Further, we know that every $gr^{V^H}_a\mathcal{M}$ is a strict $gr^{V^H}_0 \widetilde{\mathscr{D}}_X$-module. \cite[Proposition 7.3.26]{SS16} 

Recall that given a coherent $\widetilde{\mathscr{D}}_X$-module $\mathcal{M}$ equipped with a Kashiwara-Malgrange filtration $V^H_\bullet$, we have a free resolution:

\begin{prop}\label{P:free resolution}
Fix $\mathcal{M}$ a strict $\widetilde{\mathscr{D}}_X$-module that is equipped with the Kashiwara-Malgrange filtration $V^H_\bullet$ with respect to a smooth divisor $H$ on $X$. Then locally over $X$ we can find a resolution:
$$\left(\mathcal{M}^\bullet, V^H_\bullet\right) \to \left(\mathcal{M}, V^H_\bullet\right),$$
where each $\left(\mathcal{M}^i, V^H_\bullet \right)$ is a direct sum of  finite copies of $\left(\widetilde{\mathscr{D}}_X\cdot z^p , V^H[ a ]\right)$, with ${-1}\leq  a \leq 0, p\in\mathbb{Z}$. Further, the resolution is strict with respect to $V^H_\bullet$, which means it is still a resolution after taking any filtered piece of $V^H_\bullet$.

Further, if we have that the multiplication by $t$ induces an isomorphism 
$$t: \text{gr}^{V^H}_0 \mathcal{M}\to \text{gr}^{V^H}_{-1}\mathcal{M},$$
in particular, when $\mathcal{M}=\mathcal{M}[*H]$, the above resolution can be achieved with ${-1}< a \leq 0.$

If we have that applying $\widetilde{\partial}_x$ induces an isomorphism 
$$\widetilde{\partial}_x: \text{gr}^{V^H}_{-1} \mathcal{M}\to \text{gr}^{V^H}_{0}\mathcal{M}\cdot z,$$
in particular, when $\mathcal{M}=\mathcal{M}[!H]$, the above resolution can be achieved with ${-1}\leq a < 0.$
\end{prop}
\begin{proof}
Without the extra assumption, it is just \cite[3.3.9. Lemme.]{Sa88}, see also the proof of \cite[3.3.17. Proposition.]{Sa88}. With the extra assumption, it follows by the same proof.

See \cite[Proposition 9.3.4 and Proposition 9.4.2]{SS16}, for the corresponding isomorphism in the localization and dual localization case.
\end{proof}

\begin{remark*}
Actually, the existence of the free resolution as above only depends on the strictness assumption (4) in Definition \ref{D:KM-filtration on H}. See also \cite[3.3.6-3.3.9.]{Sa88}.
\end{remark*}
Now we start to consider the case that $D=D_1+...+D_n$ is a normal crossing divisor on $X$ with irreducible components $D_i$. Note that, since we only need to work locally, we can assume that all the components of $D$ are smooth. However, all the notations defined later are well defined globally, since they can be glued together, using the local data. For any 
$$\mathbf{a}=\left[a_1,...,a_n\right]\in \mathbb{R}^n,$$
 we denote 
\begin{equation}\label{E: def v fil}
    \mathbf{V}_{\mathbf{a}}^{D}\widetilde{\mathscr{D}}_X=\cap_{i} V_{a_i}^{D_i}\widetilde{\mathscr{D}}_X.
\end{equation}
For $\mathbf{a}=\mathbf{0}:=[0,...,0]$, $\mathbf{V}_{\mathbf{0}}^{D}\widetilde{\mathscr{D}}_X$ is just $\widetilde{\mathscr{D}}_{\left(X,D\right)}$.

Given a strict coherent $\widetilde{\mathscr{D}}_X$-module that possesses the Kashiwara-Malgrange filtration $V^{D_i}_\bullet$ with respect to all components $D_i$ of $D$, we define a multi-indexed Kashiwara-Malgrange filtration with respect to $D$ by 
$$\mathbf{V}^D_{\mathbf{a}}\mathcal{M}=\cap V^{D_i}_{a_i} \mathcal{M},$$
for any $\mathbf{a}=[a_1,...,a_n]\in \mathbb{R}^n.$ It is not hard to see that $\mathbf{V}^D_{\bullet}\mathcal{M}$ is a multi-indexed graded module over  $\mathbf{V}_{\bullet}^{D}\widetilde{\mathscr{D}}_X$, in the sense that
$$ \mathbf{V}^D_{\mathbf{a}}\mathcal{M}\cdot \mathbf{V}_{\mathbf{b}}^{D}\widetilde{\mathscr{D}}_X\subset \mathbf{V}^D_{\mathbf{a}+\mathbf{b}}\mathcal{M},$$
for any $\mathbf{a},\mathbf{b}\in \mathbb{R}^n.$

Given a subset $S\subset \{1,...,n\}$, we denote $D_S=\sum_{i\in S}D_i$.
For any such $D_S$, we denote 
$$\mathbf{V}^{D_S}_{\mathbf{a}}\mathcal{M}:=\cap_{i\in S} V^{D_i}_{a_i} \mathcal{M}.$$

Further, we say that $\mathbf{b}< \mathbf{a}$ if $b_i<a_i$ for all $i$. We denote 
\begin{equation}\label{E: def v fil'}
    \mathbf{V}^D_{<\mathbf{a}}\mathcal{M}:=\cup_{\mathbf{b}<\mathbf{a}} \mathbf{V}^D_{\mathbf{b}}\mathcal{M}.
\end{equation}

In general, the $n$ filtrations $V^{D_i}_\bullet$ do not behave well between each other, in the sense that, when we look $V^{D_1}_\bullet$ as a filtration on $V^{D_2}_0\mathcal{M}$, it does not have good strictness conditions as in Definition \ref{D:KM-filtration on H}. This motivates us to make the following definition, which will be essentially used in proving our main results.
\begin{definition}\label{D:V strict}
Notations as above, we say that $\mathbf{V}^D_\bullet$, the multi-indexed Kashiwara-Malgrange filtration with respect to $D$ on $\mathcal{M}$ is \emph{V-compatible} if locally, we have the following strictness relations
\begin{align}
x_i: \mathbf{V}^{D}_{\mathbf{a}}\mathcal{M}  &\xrightarrow{\sim}\mathbf{V}^{D}_{\mathbf{a}-\mathbf{1}^i}\mathcal{M}\ (a_i<0) \label{E:strict for x_i} \\ 
\widetilde{\partial}_i: gr^{V^{D^i}}_{a_i-1}\mathbf{V}^{D-D_i}_{\mathbf{a}}\mathcal{M}  &\xrightarrow{\sim} gr^{V^{D^i}}_{a_i}\mathbf{V}^{D-D_i}_{\mathbf{a}}\mathcal{M}\cdot z\ (a_i>0),\label{E:strict for p_i}
\end{align}
where
\begin{align*}
\mathbf{1}^i:=&[0,...,0,1,0,...,0],
\end{align*}
with the $1$ at the $i$-th position. Further, we require that, by replacing $\mathcal{M}$ by $\mathcal{M}[*D]$ (resp. $\mathcal{M}[!D]$), the identity (\ref{E:strict for x_i}) (resp.  (\ref{E:strict for p_i}) ) still holds when $a_i=0$.
\end{definition}

In particular, if $\mathcal{M}$ is a strict coherent $\widetilde{\mathscr{D}}_X$-module of normal crossing type with respect to $D$, then we have that $\mathbf{V}^D_\bullet$ is V-compatible, \cite[3.11. Proposition.]{Sa90}, \cite[Lemma 11.2.20.]{SS16}. More generally, it also holds if $\mathcal{M}$ is of normal crossing type with respect to $D'$, where $D'$ is a normal crossing divisor that contains $D$. Of course, if $D$ is smooth, as long as $\mathcal{M}$ underlies a mixed Hodge module, or more generally, is $\mathbb{R}$-specializable along $D$, $V^D$ is trivially V-compatible.

Note that the previous definition is different from the compatibility of the $n$ $V$-filtrations on $\mathcal{M}$, in \cite{Sa88}, which we will not use in this paper.

We introduce the notation 
$\left(\widetilde{\mathscr{D}}_X\cdot z^p , \mathbf{V}^D[ \mathbf{a} ]\right)$, to denote $\widetilde{\mathscr{D}}_X$ equipped with shifted filtrations:
$$\mathbf{V}^D_\mathbf{b}\left(\widetilde{\mathscr{D}}_X\cdot z^p , \mathbf{V}^D[ \mathbf{a} ]\right):= \left(\mathbf{V}^D_{\mathbf{b}-\mathbf{a}} \widetilde{\mathscr{D}}_X\right)\cdot z^p.$$

As Proposition \ref{P:free resolution}, we have a similar free resolution:
\begin{prop}\label{P:free resolution NC}
Fix $\mathcal{M}$, a strict coherent $\widetilde{\mathscr{D}}_X$-module admitting a V-compatible multi-indexed Kashiwara-Malgrange filtration with respect to $D$, we can find a resolution:
$$\left(\mathcal{M}^\bullet, \mathbf{V}^D_\bullet\right) \to \left(\mathcal{M}, \mathbf{V}^D_\bullet\right),$$
where each $\left(\mathcal{M}^i, \mathbf{V}^D_\bullet \right)$ is a direct sum of  finite copies of $\left(\widetilde{\mathscr{D}}_X\cdot z^p , \mathbf{V}^D[ \mathbf{a} ]\right)$, with $\mathbf{-1}\leq  \mathbf{a} \leq \mathbf{0}, p\in\mathbb{Z}$. Further, the resolution is strict with respect to $\mathbf{V}^D_\bullet$, which means it is still a resolution after taking any filtered piece of $\mathbf{V}^D_\bullet$.

Further, if we have $\mathcal{M}=\mathcal{M}[*D]$  (resp. $\mathcal{M}=\mathcal{M}[!D]$ ), the above resolution can be achieved with $\mathbf{-1}< \mathbf{a} \leq \mathbf{0}$ (resp. $\mathbf{-1}\leq \mathbf{a} < \mathbf{0}$).
\end{prop}

\begin{proof}
Let $\Lambda=\{(\lambda_1,...,\lambda_n)\in \mathbb{R}^n \mid -1\leq \lambda_i\leq 0, gr^{V^{D_i}}_{\lambda_i}\mathcal{M}\neq 0, \text{for all } i=1,...,n\}$, which is a finite set.

As in the proof of \cite[3.3.9. Lemme.]{Sa88}, we only need to show that locally, the natural morphism:
$$u:\oplus_{p\in \mathbf{Z}, \mathbf{a}\in \Lambda}F_p\mathbf{V}^D_\mathbf{a} \mathcal{M}\otimes \left(\widetilde{\mathscr{D}}_X\cdot z^p, \mathbf{V}^D[ \mathbf{a} ]\right)\to \mathcal{M}$$
is strictly surjective, i.e. it is still surjective after taking the $\mathbf{V}^D_{\mathbf{a}}$ filtered piece, for any $\mathbf{a}\in \mathbb{R}^n$. Of course, it is the case for $\mathbf{-1}\leq \mathbf{a} \leq \mathbf{0}$. Due to the strictness condition (\ref{E:strict for x_i}) and (\ref{E:strict for p_i}), we have that $\mathbf{V}^D_\mathbf{a} \mathcal{M}$ is fully determined by  $\mathbf{V}^D_\mathbf{a'}\mathcal{M}$, where $\mathbf{a'} \in \Lambda$, such that $\mathbf{a}-\mathbf{a'}\in \mathbb{Z}^n$. In particular, induced by $u$, we have the surjection
$$\oplus_{p\in \mathbf{Z}} F_p\mathbf{V}^D_\mathbf{a'} \mathcal{M}\otimes \mathbf{V}^D_{\mathbf{a}-\mathbf{a'}} \widetilde{\mathscr{D}}_X\cdot z^p \to \mathbf{V}^D_\mathbf{a}\mathcal{M},$$
for general $\mathbf{a}.$ 
\end{proof}

\section{Comparison Theorem}
Fix a smooth variety $X$ and a normal crossing divisor $D=D_1+...+D_n$ on $X$ as in the previous sections. We start with the following vanishing.
\begin{lemma}\label{L:vanishing by torsion freeness}
Given a strict $\widetilde{\mathscr{D}}_{\left(X,D\right)}$-module $\mathcal{M}$. Let $D_S=D-D_1$ and $x_1$ be a local function that defines $D_1$. Assume that $\text{Gr}^F \mathcal{M}$ is torsion free with respect to $x_1$, then we have 
$$\mathcal{H}^i\left(\mathcal{M}\otimes^{\mathbf{L}}_{\widetilde{\mathscr{D}}_{\left(X,D\right)}} \widetilde{\mathscr{D}}_{\left(X,D_S\right)}\right)=0,$$
for $i\neq 0$.
\end{lemma}

\begin{proof}
Recall the canonical left resolution of $\mathcal{M}$ in (\ref{E:M otimes spencer}). Note that we have
$$\left(\mathcal{M}\otimes\wedge^{i}\widetilde{\mathcal{T}}_{\left(X,D\right)}\otimes \widetilde{\mathscr{D}}_{\left(X,D\right)}\right)_\text{triv}\otimes^{\mathbf{L}}_{\widetilde{\mathscr{D}}_{\left(X,D\right)}} \widetilde{\mathscr{D}}_{\left(X,D_S\right)}
\simeq \left(\mathcal{M}\otimes\wedge^{i}\widetilde{\mathcal{T}}_{\left(X,D\right)}\otimes \widetilde{\mathscr{D}}_{\left(X,D_S\right)}\right)_\text{triv},$$
by the associativity of the tensor product. 
In particular (\ref{E:M otimes spencer}) is a resolution with acyclic objects with respect to the functors $\otimes^{\mathbf{L}}_{\widetilde{\mathscr{D}}_{\left(X,D\right)}} \widetilde{\mathscr{D}}_{\left(X,D_S\right)}$ and $\widetilde{\text{Gr}}^F$. Hence, after apply $\text{Gr}^F$ onto the resolution, we get a resolution of $\text{Gr}^F\mathcal{M}$, by acyclic objects with respect to $\otimes^{\mathbf{L}}_{\text{Gr}^F\widetilde{\mathscr{D}}_{\left(X,D\right)}} \text{Gr}^F\widetilde{\mathscr{D}}_{\left(X,D_S\right)}$. Now, we only need to show 
$$\mathcal{H}^i\left(\text{Gr}^F\mathcal{M}\otimes^{\mathbf{L}}_{\text{Gr}^F\widetilde{\mathscr{D}}_{\left(X,D\right)}} \text{Gr}^F\widetilde{\mathscr{D}}_{\left(X,D_S\right)}\right)=0,$$
for all $i\neq 0$. 

Let $\partial '_1,...,\partial '_{d_X}$
be the induced elements in $\text{Gr}^F_1 \widetilde{\mathscr{D}}_X$. Locally, we have isomorphisms
\begin{align*}
    \text{Gr}^F\widetilde{\mathscr{D}}_{\left(X,D\right)}&\simeq\mathcal{O}_X[\eta_1,...,\eta_{n},\xi_{n+1},...,\xi_{d_X}],\\
    \text{Gr}^F\widetilde{\mathscr{D}}_{\left(X,D_S\right)}&\simeq\mathcal{O}_X[\xi_1,\eta_2,...,\eta_{n},\xi_{n+1},...,\xi_{d_X}],
\end{align*}
by identifying $\xi_i=\partial '_i$, $\eta_i= x_i\partial'_i$.
Note that we have the short exact sequence 
$$0\to \text{Gr}^F\widetilde{\mathscr{D}}_{\left(X,D\right)}[\xi_1]\xrightarrow{\cdot(x_1\xi_1-\eta_1)}  \text{Gr}^F\widetilde{\mathscr{D}}_{\left(X,D\right)}[\xi_1]\to \text{Gr}^F\widetilde{\mathscr{D}}_{\left(X,D_S\right)}\to 0 
$$
as a free resolution of $\text{Gr}^F\widetilde{\mathscr{D}}_{\left(X,D_S\right)}$ as $\text{Gr}^F\widetilde{\mathscr{D}}_{\left(X,D\right)}$ modules. Now we are only left to show that $\text{Gr}^F\mathcal{M}[\xi_1]\xrightarrow{\cdot(x_1\xi_1-\eta_1)} \text{Gr}^F\mathcal{M}[\xi_1]$ is injective.
For any $m\in \text{Gr}^F\mathcal{M}[\xi_1]$, by considering the degree of $\xi_1$, if $m\cdot(x_1\xi_1-\eta_1)=0$, we must have $m\cdot x_1=0$. Since $\text{Gr}^F\mathcal{M}$ is torsion free, with respect to $x_1$, we have $m=0$.

\end{proof}

Now we consider $\mathcal{M}$ as a mixed Hodge module on $X$. Given a divisor $D$ with normal crossing support on $X$, we denote 
$$\mathcal{M}[*D]=i_*i^{-1}\mathcal{M},$$
and 
$$\mathcal{M}[!D]=i_!i^{-1}\mathcal{M}$$
being the localization and dual localization of $\mathcal{M}$ on $X\setminus D$ \cite[2.11 Proposition]{Sa90}, \cite[Chapter 9]{SS16}.
Note that both $\mathcal{M}[*D]$ and $\mathcal{M}[!D]$ are mixed Hodge modules, in particular, strict coherent $ \widetilde{\mathscr{D}}_X$-modules.

\begin{proof}[Proof of Theorem \ref{T:Comparison Theorem}]
Since $\mathcal{M}[*D]=\mathcal{M}[*D][*D_I]$, we can assume $D_I=D$ without lose of generality. Further, by the previous lemma and the induction on the number of components of $D$, to show (\ref{E:* comparison}), we only need to show the case that $D_S=D-D_1$, 
\begin{equation}\label{E:one tensor}
\mathbf{V}^{D}_{\mathbf{0}}\left( \mathcal{M}[*D]\right)\otimes_{\widetilde{\mathscr{D}}_{\left(X,D\right)}} \widetilde{\mathscr{D}}_{\left(X,D_S\right)} \simeq  \mathbf{V}^{D_{S}}_{\mathbf{0}}\left( \mathcal{M}[*D]\right).
\end{equation}
It follows, by a similar argument in \cite[Lemma 3.1.2]{Moc15}, by using the properties in Definition \ref{D:V strict}, so we omit the details here. See also \cite[Proposition 9.3.4]{SS16}.

(\ref{E:! comparison}) follows similarly, by \cite[Lemma 3.1.11]{Moc15}. See also \cite[Proposition 9.4.2]{SS16}.
\end{proof}

\begin{remark*}
Actually, Theorem \ref{T:Comparison Theorem} can also be proved by using the free resolution in Proposition \ref{P:free resolution NC}. More precisely, for the $\mathcal{M}[*D]$ case, we can get a resolution with each term being a direct sum of finite copies of $\left(\widetilde{\mathscr{D}}_X\cdot z^p , \mathbf{V}^D[ \mathbf{a} ]\right)$, with $\mathbf{-1}<  \mathbf{a} \leq \mathbf{0}, p\in\mathbb{Z}$. Note that in this case, we have 
$$\mathbf{V}^D_\mathbf0\left(\widetilde{\mathscr{D}}_X\cdot z^p , \mathbf{V}^D[ \mathbf{a} ]\right)=\widetilde{\mathscr{D}}_{(X,D)},$$ 
hence by taking the $\mathbf{V}^D_\mathbf0$ piece, we also get a free $\widetilde{\mathscr{D}}_{(X,D)}$-module resolution. A similar argument will be used in proving Theorem \ref{T:dual} in Section 7.
\end{remark*}

Applying the Spencer functor, we can recover the classical logarithmic comparison theorem.

\begin{proof}[Proof of Theorem \ref{T:Comparison Theorem Sp}]
By the definition of the Spencer functor, we have
\begin{align*}
&\textsl{Sp}_{\left(X,D\right)}\left(\mathcal{M}[*\log D]\right)\\
\simeq&\mathbf{V}^{D}_{{0}}\left(\mathcal{M}[*D]\right)\otimes^{\mathbf{L}}_{\widetilde{\mathscr{D}}_{\left(X,D\right)}} \widetilde{\mathscr{O}}_X\\
\simeq &\mathbf{V}^{D}_{{0}}\left(\mathcal{M}[*D]\right)\otimes^{\mathbf{L}}_{\widetilde{\mathscr{D}}_{\left(X,D\right)}}\left(\widetilde{\mathscr{D}}_{X}\otimes^{\mathbf{L}}_{\widetilde{\mathscr{D}}_{X}} \widetilde{\mathscr{O}}_X\right)\\
\simeq&\mathcal{M}[*D]\otimes^{\mathbf{L}}_{\widetilde{\mathscr{D}}_{X}} \widetilde{\mathscr{O}}_X\\
\simeq&\textsl{Sp}_X\left(\mathcal{M}[*D]\right).
\end{align*}
The third identity is due to associativity of tensor products and the comparison theorem, by taking $D_I=D, $ and $D_S=0$.

The dual localization case follows similarly.
\end{proof}

\section{Pushforward functor}
Fix $\mathcal{M}$, a strict coherent $\widetilde{\mathscr{D}}_X$-module, admitting a V-compatible multi-indexed Kashiwara-Malgrange filtration with respect to a normal crossing divisor $D$, and
$$D=r_1D_1+...+r_nD_n $$
be the decomposition of reduced and irreducible components. Denote $\mathbf{r}=\left[r_1,...,r_n\right]$.

Working locally, let $i:\left(X,D\right)\to \left(Y,H\right)$ be the graph embedding given by the local function $y=x_1^{r_1}...x_n^{r_n}$, where each $x_i$ is a local function that defines the divisor $D_i$ on $X$. We first build a relation between the Kashiwara-Malgrange filtration $\mathbf{V}^D_\bullet$ on $\mathcal{M}$ and $V^H_\bullet$ on $i_+\mathcal{M}$.

Let $\widetilde{\partial}_{1},...,\widetilde{\partial}_{n}, \widetilde{\partial}_y\in \widetilde{\mathscr{D}}_Y$ that are mutually commutative and satisfying  
\begin{align*}
\left[\partial_{i}, x_i\right]=&[\partial_y, y]=1,\\
\left[\partial_{i},x_j \right]=&[\partial_{i}, y]=[\partial_y,x_i]=0. \text{for} \ i\neq j
\end{align*}
Let $u=x_1^{r_1}...x_n^{r_n}-y$, and consider $x_1,...,x_n, u$ being a local coordinate system on $Y$. We can get $\widetilde{\partial}_{x_1},...,\widetilde{\partial}_{x_n}, \widetilde{\partial}_u\in \widetilde{\mathscr{D}}_Y$ that satisfying 
\begin{align*}
\left(m\otimes \delta\right)\partial_{x_i}=&m\partial_{x_i}\otimes \delta,\\
\left(m\otimes \delta\right)\partial_{u}=&m\otimes \delta\partial_{u},
\end{align*}
and we can further require that
\begin{align*}
\left[\partial_{x_i}, x_i\right]=&[\partial_u, u]=1,\\
\left[\partial_{x_i},x_j \right]=&[\partial_{x_i}, u]=[\partial_u, x_i]=0, \text{for} \ i\neq j.
\end{align*}
We can write $i_\#\mathcal{M}=\oplus_{k\in \mathbb{N}} \mathcal{M}\otimes \delta \widetilde{\partial}_u^k$ as a $\widetilde{\mathscr{D}}_Y=\widetilde{\mathscr{D}}_X[u]\langle\widetilde{\partial}_u\rangle$-module. By changing of coordinates, we have the following relations:
\begin{align*}
\left(m\otimes \delta\right)\widetilde{\partial}_y^k&=m\otimes \delta\widetilde{\partial}_u^k,\\
\left(m\otimes \delta\right)\widetilde{\partial}_i&=\left(m\widetilde{\partial}_{x_i}\right)\otimes \delta -\left(m r_ix_1^{r_1}...x_n^{r_n}/x_i \right)\otimes \delta \widetilde{\partial}_u,\\
\left(m\otimes \delta\right)y&=\left(m x_1^{r_1}...x_n^{r_n}\right)\otimes \delta,\\
\left(m\otimes \delta\right)\widetilde{\mathscr{O}}_X&=\left(m\widetilde{\mathscr{O}}_X\right)\otimes \delta,
\end{align*}
and with the usual commutation rules. 

\begin{lemma}\label{L:induced KM-filtration}
Notations as above, we can get the Kashiwara-Malgrange filtration $V^{H}_\bullet$ on $i_+\mathcal{M}$ from $\mathbf{V}^D_\bullet$ on $\mathcal{M}$ by:
\begin{equation}\label{E:a<0 induced KM-filtration}
V^{H}_a i_+\mathcal{M}=\left(\mathbf{V}^{D}_{a\cdot \mathbf{r}}\mathcal{M}\otimes \delta\right)\cdot \widetilde{\mathscr{D}}_{(Y,H)},\  \text{for}\ a< 0.
\end{equation}
For $a \geq 0$, we can get it inductively by
$$
V^{H}_ai_+\mathcal{M}\cdot z=V^{H}_{<a}\left(i_+\mathcal{M}\right)\cdot z +V^{H}_{a-1}\left(i_+\mathcal{M}\right)\widetilde{\partial}_y.
$$
If we further have that 
\begin{equation}\label{E:extra condition}
    x_i:V^{D_i}_0\mathcal{M}\to V^{D_i}_{-1}\mathcal{M}
\end{equation}
is an isomorphism for every component $D_i$, in particular when $\mathcal{M}=\mathcal{M}[*D]$, we have (\ref{E:a<0 induced KM-filtration}) still holds when $a=0$.
\end{lemma}
\begin{proof}
The proof is similar to the proof in \cite[Theorem 11.3.2]{SS16}. However, since our settings are a little bit more general, we spell out the details here.

We need to check the filtration we defined above satisfies the Condition (1), (2), (3) in Definition \ref{D:KM-filtration on H}. We only check the case with the extra condition here. The other case follows similarly.

For Condition (1), if $a\leq 0$ and if $m_1,...,m_l$ generate $\mathbf{V}^{D}_{a\cdot \mathbf{r}}\mathcal{M}$ over $\widetilde{\mathscr{D}}_{\left(X, D^X\right)}$, then $m_1\otimes \delta,...,m_l\otimes \delta$ generate $V^{H}_ai_+\mathcal{M}$ over $\widetilde{\mathscr{D}}_{\left(Y, D^Y\right)}$, due to the fact that we have the relation
$$\left(mx_i\widetilde{\partial}_{x_i}\right)\otimes \delta=\left(m\otimes \delta\right)\left(x_i\widetilde{\partial}_i+r_iy\widetilde{\partial}_y\right).$$
Hence, we can conclude that, for every $a$, $V^{H}_ai_+\mathcal{M}$ is coherent over $\widetilde{\mathscr{D}}_{\left(Y, D^Y\right)}$.
To show that 
$$\cup_a V^{H}_a i_+\mathcal{M}=i_+\mathcal{M},$$
we only need to show that 
$$\mathcal{M}\otimes \delta\subset \cup_a V^{H}_a i_+\mathcal{M}.$$
 Note that 
$$\left(\mathbf{V}^{D}_{\mathbf{0}} \mathcal{M}\right)\cdot \widetilde{\mathscr{D}}_X =\mathcal{M},$$
 which is due to the assumption that $\mathcal{M}$, hence $\mathcal{M}\left[*D\right]$ is of normal crossing type, hence the strictness condition (\ref{E:strict for p_i}). Then, we can use the relation
$$\left(m\widetilde{\partial}_{x_i}\right)\otimes \delta=\left(m\otimes \delta\right)\widetilde{\partial}_i+\left(m r_ix_1^{r_1}...x_n^{r_n}/x_i \right)\otimes \delta \widetilde{\partial}_u,
$$
to get that 
$$\cup_a V^{H}_a i_+\mathcal{M}\supset \left(\mathbf{V}^{D}_{\mathbf{0}} \mathcal{M}\otimes \delta\right)\cdot \widetilde{\mathscr{D}}_{Y}\supset \mathcal{M}\otimes \delta.$$

For Condition (2), we have that, for any $a$,
$$\left(V^{H}_a i_+\mathcal{M}\right) y\subset V^{H}_{a-1} i_+\mathcal{M},$$
which can be checked by using the relation 
$$\left(m\otimes \delta\right)y=\left(m x_1^{r_1}...x_n^{r_n}\right)\otimes \delta.$$
Further, the equality for $a<0$ can also be deduced from the corresponding properties on $\mathcal{M}$.
We are left to show 
$$\left(V^{H}_{a-1} i_+\mathcal{M}\right)\widetilde{\partial}_y \subset V^{H}_{a} i_+\mathcal{M}.$$
When $a> 0$, it follows by definition. When $a\leq 0$, for any $m\otimes \delta\in V^{H}_a i_+\mathcal{M}$,
considering the relation \cite[11.2.19]{SS16} and our extra assumption (\ref{E:extra condition}), we obtain
$$\left(\mathbf{V}^{D}_{\mathbf{a}}\mathcal{M}\right)x_i= \mathbf{V}^{D}_{\mathbf{a}-\mathbf{1}^i}\mathcal{M},$$
for $a\leq 0$. Hence, for any $m'\otimes \delta\in V^{H}_{a-1}i_+\mathcal{M}$ we can write 
$$m'\otimes \delta=\left(m r_ix_1^{r_1}...x_n^{r_n}/x_i \right)\otimes \delta,$$
for some $m\in \mathbf{V}^{D}_{a\cdot\mathbf{r}-\mathbf{1}^i}\mathcal{M}.$
Now using the relation 
$$\left(m\otimes \delta\right)\widetilde{\partial}_i=\left(m\widetilde{\partial}_{x_i}\right)\otimes \delta -\left(m r_ix_1^{r_1}...x_n^{r_n}/x_i \right)\otimes \delta \widetilde{\partial}_u,$$
we obtain
$$\left(m'\otimes \delta\right)\partial_y=\left(m\widetilde{\partial}_{x_i}\right)\otimes \delta-\left(m\otimes \delta\right)\widetilde{\partial}_i\in V^{H}_{a} i_+\mathcal{M},$$
which concludes the proof of the Condition (2).

For Condition (3), by direct computation, we have
$$
\left(m\otimes \delta\right)(y\widetilde{\partial}_y-az)=\left(m \frac{1}{r_i}\left(x_i\widetilde{\partial}_{x_i}-r_iaz\right)\otimes \delta\right)-\left(x_i\frac{1}{r_i}m\otimes \delta\right)\widetilde{\partial}_i.
$$
Using this relation and assuming $a\leq 0$, we can get that, if 
$$V^{D_i}_{r_ia} \mathcal{M}\left(x_i\widetilde{\partial}_{x_i}-r_iaz\right)^{\nu_{r_ia}}\subset V^{D_i}_{<r_ia} \mathcal{M},$$
we have that 
$$\left(\mathbf{V}^{D}_{a\cdot \mathbf{r}} \mathcal{M}\otimes \delta\right)\left(y\widetilde{\partial}_{y}-az\right)^{\nu_{r_ia}}\subset \left(\mathbf{V}^{D}_{a\cdot \mathbf{r}-\bar{\epsilon}^i} \mathcal{M}\otimes \delta\right)\cdot\widetilde{\mathscr{D}}_{\left(Y,H\right)},$$
where $\bar{\epsilon}^i:=[0,...,0,\epsilon,0,...,0],$ where $0<\epsilon\ll 1$ is located at the $i$-th position. Denote $\mu_a=\sum_i \nu_{r_ia}$, do the computation above inductively, we obtain 
$$\left(\mathbf{V}^{D}_{a\cdot \mathbf{r}} \mathcal{M}\otimes \delta\right)\left(y\widetilde{\partial}_{y}-az\right)^{\mu_{a}}\subset \left(\mathbf{V}^{D}_{<a\cdot \mathbf{r}} \mathcal{M}\otimes \delta\right)\cdot\widetilde{\mathscr{D}}_{\left(Y,H\right)}.$$
Hence $y\widetilde{\partial}_{y}-az$ is nilpotent on $\text{gr}^{V^{H}}_{a} i_+\mathcal{M}$, for $a\leq 0$.
When $a> 0$, we can get the same conclusion by induction and using the relation 
$$\widetilde{\partial}_y\left(y\widetilde{\partial}_y-az\right)=\left(y\widetilde{\partial}_y-\left(a-1\right)z\right)\widetilde{\partial}_y.$$
\end{proof}

\begin{lemma}\label{L:graph embedding no higher image}
Notations as above. Fix a strict $\widetilde{\mathscr{D}}_{\left(X,D\right)}$-module $\mathcal{N}$. Assume $\text{Gr}^F\mathcal{N}$ is torsion-free with respect to every $x_i$, then we have 
$$i_\# \mathcal{N}\simeq \mathcal{H}^0i_\# \mathcal{N}.$$
\end{lemma}
\begin{proof}
Let $D'_i$ be the divisors on $Y$ that are defined by the the local functions $x_i$. Denote $D':=D'_1+...+D'_n$. We can decompose $i:\left(X,D\right)\to \left(Y,H\right)$ into $j:\left(X,D\right)\to \left(Y, D'+H\right)$ and $id: \left(Y, D'+H\right)\to \left(Y,H\right)$. By (\ref{E:direct image}), we have 
\begin{align*}
j_\# \mathcal{N}&\simeq j_*\left(\mathcal{N}\otimes^{\mathbf{L}}_ {\widetilde{\mathscr{D}}_{\left(X,D\right)}}\left(\widetilde{\mathscr{O}}_X\otimes_{f^{-1}\widetilde{\mathscr{O}}_Y}f^{-1}\widetilde{\mathscr{D}}_{\left(Y,D'+H\right)}\right) \right).
\end{align*}
However, it is straight forward to check that the natural map
$$dj:j^*\Omega_Y\left(\text{log }(D'+H)\right)\to \Omega_X\left(\text{log }D\right)$$
is surjective. Hence its dual
$$\partial j:\mathcal{T}_{\left(X,D\right)}\to j^*\mathcal{T}_{\left(Y,D'+H\right)}$$
has a locally free cokernal and we denote it by $\mathcal{K}$. Then we have 
$$\widetilde{\mathscr{O}}_X\otimes_{f^{-1}\widetilde{\mathscr{O}}_Y}f^{-1}\widetilde{\mathscr{D}}_{\left(Y,D'+H\right)}\simeq\widetilde{\mathscr{D}}_{\left(X,D\right)} \otimes \text{Sym}\left(\widetilde{\mathcal{K}}\cdot z\right).$$
In particular, we have that $\widetilde{\mathscr{O}}_X\otimes_{f^{-1}\widetilde{\mathscr{O}}_Y}f^{-1}\widetilde{\mathscr{D}}_{\left(Y,D'+H\right)}$ is locally free over $\widetilde{\mathscr{D}}_{\left(X,D\right)}$, hence we obtain
$$j_\# \mathcal{N}\simeq \mathcal{H}^0j_\# \mathcal{N}.$$ 
For $id: \left(Y, D'+H\right)\to Y$ part, by (\ref{E:direct image}), we have 
\begin{align*}
id_\# \left(j_\#\mathcal{N}\right)&\simeq \left(j_\#\mathcal{N}\right)\otimes^{\mathbf{L}}_ {\widetilde{\mathscr{D}}_{\left(Y,D'+H\right)}}\widetilde{\mathscr{D}}_{\left(Y,H\right)}.
\end{align*}
However, by the assumption that $\text{Gr}^F\mathcal{N}$ is torsion-free with respect to $x_i$, hence so is $\text{Gr}^F j_\#\mathcal{N}$. By applying Lemma \ref{L:vanishing by torsion freeness} inductively, we obtain
$$id_\# \left(j_\#\mathcal{N}\right)\simeq \mathcal{H}^0 id_\# \left(j_\#\mathcal{N}\right).$$
\end{proof}

\begin{lemma}\label{L:graph embedding case}
Notations as in Lemma \ref{L:induced KM-filtration}. We have
\begin{align*}
i_{\#}\mathbf{V}^D_{a\cdot \mathbf{r}} \mathcal{M}\simeq V^H_ai_+\mathcal{M} \text{ for } a<0.
\end{align*}
Further, if the morphism 
\begin{equation*}
x_i:V^{D_i}_0\mathcal{M}\to V^{D_i}_{-1}\mathcal{M}
\end{equation*}
is an isomorphism for every component $D_i$, e.g. $\mathcal{M}=\mathcal{M}[*D]$, the identity still holds when $a=0$
\end{lemma}
\begin{proof}
Since $\text{Gr}^F\mathbf{V}^D_{a\cdot \mathbf{r}} \mathcal{M}$ is $x_i$ torsion free for $a<0$ (resp. $a\leq 0$ with the extra assumption,) due to the strictness condition (\ref{E:strict for x_i}), applying the previous lemma, we get 
$$i_{\#}\mathbf{V}^D_{a\cdot \mathbf{r}} \mathcal{M}\simeq \mathcal{H}^0 i_{\#}\mathbf{V}^D_{a\cdot \mathbf{r}} \mathcal{M},$$
for $a<0$ (resp. $a \leq 0$).

We have 
$$\mathcal{H}^0 i_{\#}\mathbf{V}^D_{a\cdot \mathbf{r}} \mathcal{M} \simeq i_* \left(\mathbf{V}^D_{a\cdot \mathbf{r}} \mathcal{M} \otimes_{\widetilde{\mathscr{D}}_{\left(X,D\right)}}\left(\widetilde{\mathscr{O}}_X\otimes_{i^{-1}\widetilde{\mathscr{O}}_Y}i^{-1}\widetilde{\mathscr{D}}_{\left(Y,H\right)}\right)\right),$$
and 
$$i_+\mathcal{M}\simeq i_* \left(\mathcal{M} \otimes_{\widetilde{\mathscr{D}}_{X}}\left(\widetilde{\mathscr{O}}_X\otimes_{i^{-1}\widetilde{\mathscr{D}}_Y}\widetilde{\mathscr{D}}_{Y}\right)\right).$$
There is a natural morphism 
$$
i_* \left(\mathbf{V}^D_{a\cdot \mathbf{r}} \mathcal{M} \otimes_{\widetilde{\mathscr{D}}_{\left(X,D\right)}}\left(\widetilde{\mathscr{O}}_X\otimes_{i^{-1}\widetilde{\mathscr{D}}_Y}\widetilde{\mathscr{D}}_{\left(Y,H\right)}\right)\right)\to 
i_* \left(\mathcal{M} \otimes_{\widetilde{\mathscr{D}}_{X}}\left(\widetilde{\mathscr{O}}_X\otimes_{i^{-1}\widetilde{\mathscr{D}}_Y}\widetilde{\mathscr{D}}_{Y}\right)\right),$$
which is an injection due the torsion free condition on $\text{Gr}^F\mathbf{V}^D_{a\cdot \mathbf{r}} \mathcal{M}$, and its image is exactly
$$\left(\mathbf{V}^D_{a\cdot \mathbf{r}} \mathcal{M}\otimes \delta\right)\cdot\widetilde{\mathscr{D}}_{\left(Y,H\right)}.$$
We can conclude the proof, by using Lemma \ref{L:induced KM-filtration}.
\end{proof}

Fix a projective morphism between two smooth varieties $f: X\to Y$. Given $D^X, D^Y$ two reduced divisors with normal crossing support on $X, Y$, respectively, recall that we say $f:\left(X,D^X\right)\to \left(Y, D^Y\right)$ is a projective morphism of log-smooth pairs if we further have that $f^{-1}D^Y\subset D^X$.

\begin{prop}
Notations as above, assume that $f^{-1}D^Y= D^X$. Fix a mixed Hodge module $\mathcal{M}$ on $X$. Then we have 
\begin{align*}
\mathcal{H}^if_+ \mathcal{M}\left[*D^X\right]=&\left(\mathcal{H}^if_+ \mathcal{M}\right)[*D^Y],\\
\mathcal{H}^if_+ \mathcal{M}\left[!D^X\right]=&\left(\mathcal{H}^if_+ \mathcal{M}\right)[!D^Y].
\end{align*}
\end{prop}

\begin{proof}
By \cite[2.11. Proposition]{Sa90}, we only need to check the identities above at the level of perverse sheaves, which is the case by considering the following commutative diagram:
\begin{center}
\begin{tikzcd}
X\setminus D^X\arrow[r,hook]{r}{i_X}\arrow{d}{f'}
&X\arrow{d}{f}\\
Y\setminus D^Y\arrow[r,hook]{r}{i_Y}
&Y.
\end{tikzcd}
\end{center}
Noticing that $i_X$ and $i_Y$ are affine and $f$ and $f'$ are projective, we have 
\begin{align*}
\mathbf{R}f_*\mathbf{R}i_{X*}K=&\mathbf{R}i_{Y*}\mathbf{R}f'_*K,\\
\mathbf{R}f_*\mathbf{R}i_{X!}K=&\mathbf{R}i_{Y!}\mathbf{R}f'_*K,
\end{align*}
for any perverse sheaf $K$ on $X\setminus D^X$.
\end{proof}

\begin{proof}[Proof of Theorem \ref{T:direct image}]
We only show the proof of the first identity here. The second one follows similarly. 
Consider the morphism of log pairs, induced by the identity map 
$$id: \left(X, D^X\right)\to \left(X, f^{-1}D^Y\right),$$
and its induced direct image of $V^{D^X}_{\mathbf{0}} \mathcal{M}[*D^X]$. We can forget the logarithmic structure on those components of $D^X-f^{-1}D^Y$. More precisely, we have 
\begin{align*}
id_\#V^{D^X}_{\mathbf{0}} \mathcal{M}[*D^X] \simeq& V^{D^X}_{\mathbf{0}} \mathcal{M}[*D^X]\otimes^\mathbf{L}_{\widetilde{\mathscr{D}}_{\left(X,D^X\right)}}\widetilde{\mathscr{D}}_{\left(X,f^{-1}D^Y\right)}\\
\simeq& V^{f^{-1}D^Y}_{\mathbf{0}}\mathcal{M}[*D^X].
\end{align*}
The second identity is due to the comparison theorem. Hence, without loss of generality, we can further assume that $D^X=f^{-1} D^Y$.

Work locally over $Y$, and consider the following commutative diagram
\begin{center}
\begin{tikzcd}
\left(X,D^X\right)\arrow{r}{i}\arrow{d}{f}&
\left(S,H^S\right)\arrow{d}{g}\\
\left(Y,D^Y\right)\arrow{r}{j}&
\left(T,H^T\right),
\end{tikzcd}
\end{center}
where 
$$i:\left(X,D^X\right)\to \left(S,H^S\right)$$
is the graph embedding given by a local function $f^*y$, where $y$ is a local function on $Y$ that defines $D^Y$, and 
$$j:\left(Y,D^Y\right)\to \left(T,H^T\right),$$
 is the graph embedding given by $y$. Let
$$g:\left(S,H^S\right)\to \left(T,H^T\right)$$
be the naturally induced morphism, and we have $g^*H^T=H^S$.
By Lemma \ref{L:graph embedding case}, we have 
$$i_\# \mathbf{V}^{D^X}_{\mathbf{0}} \left(\mathcal{M}\left[*D^X\right]\right)\simeq V^{H^S}_0 i_+\left(\mathcal{M}[*H^S]\right).$$
Now by \cite[Theorem 7.8.5]{SS16}, \cite[3.3.17]{Sa88} we have 
$$\mathcal{H}^ig_\#V^{H^S}_0 i_+\left(\mathcal{M}\left[*H^S\right]\right)\simeq V^{H^T}_0 \left((\mathcal{H}^i \left(g\circ i\right)_+\mathcal{M})\left[*H^T\right]\right).$$
Hence we obtain
\begin{equation}\label{E:one side push}
\left(g\circ i\right)_\# \mathbf{V}^{D^X}_{\mathbf{0}}\left( \mathcal{M}\left[*D^X\right]\right)\simeq V^{H^T}_0 \left((\mathcal{H}^i \left(g\circ i\right)_+\mathcal{M})\left[*H^T\right]\right).
\end{equation}

Note that, since both $D^Y$ and $H^T$ are smooth and $j^*H^T=D^Y$, 
by a local computation, we have that
$$dj:j^*\Omega_T\left(log D^T\right)\to \Omega_Y\left(log D^Y\right)$$
is surjective. Hence, as in Lemma \ref{L:graph embedding no higher image}, we have that $\widetilde{\mathscr{D}}_{\left(Y,D^Y\right)\to \left(T,H^T\right)}$ is a flat left $\widetilde{\mathscr{D}}_{\left(Y,D^Y\right)}$-module. Hence 
$$j_\#\mathcal{H}^i f_\# \left(\mathbf{V}^{D^X}_{\mathbf{0}} \mathcal{M}\left[*D^X\right]\right)\simeq \mathcal{H}^0j_\#\mathcal{H}^i f_\# \left(\mathbf{V}^{D^X}_{\mathbf{0}} \mathcal{M}\left[*D^X\right]\right).$$
By the degeneration of Leray spectral sequence and (\ref{E:one side push}), we obtain
$$j_\#\mathcal{H}^i f_\# \left(\mathbf{V}^{D^X}_{\mathbf{0}} \mathcal{M}\left[*D^X\right]\right)\simeq V^{H^T}_0 \left(\left(\mathcal{H}^i \left(j\circ f\right)_+\mathcal{M}\right)\left[*H^T\right]\right).$$
Apply the previous lemma again, we have
$$j_\#\mathcal{H}^i f_\# \left(\mathbf{V}^{D^X}_{\mathbf{0}} \mathcal{M}\left[*D^X\right]\right)\simeq j_\#\left(V^{D^Y}_0 \left(\mathcal{H}^i f_+\mathcal{M}\right)\left[*D^Y\right]\right).$$

Since $D^Y$ is smooth, by changing of coordinates, locally we have 
$$\widetilde{\mathscr{D}}_{\left(Y,D^Y\right)\to \left(T,H^T\right)}\simeq \widetilde{\mathscr{D}}_{\left(Y,D^Y\right)}\otimes \widetilde{\mathscr{O}}_Y[\widetilde{\partial}_u].,$$
which implies that $j_\#$ is a fully faithful functor. Hence we can conclude that 
$$\mathcal{H}^i f_\# \left(\mathbf{V}^{D^X}_{\mathbf{0}} \mathcal{M}\left[*D^X\right]\right)\simeq V^{D^Y}_0\left( \left(\mathcal{H}^i f_+\mathcal{M}\right)\left[*D^Y\right]\right).$$ 
In particular, it is a strict $\widetilde{\mathscr{D}}_{\left(Y,D^Y\right)}$-module.
\end{proof}
\begin{remark*}
$D^Y$ being smooth is important in Theorem \ref{T:direct image}. If $D^Y$ is not smooth, Theorem \ref{T:direct image} will not be correct for $f$ in general, since $\mathcal{H}^if_+\mathcal{M}$ may not admit V-compatible Kashiwara-Malgrange filtrations with respect to $D^Y$. Even with this extra assumption, it might still not be enough, since we now need to consider several Kashiwara-Malgrange filtrations, with respect to each component of $D^Y$, induced on the complex $f_+\mathcal{M}$. Then we need to consider the strictness of the multi-filtered complex as in \cite[8.3]{SS16}. The previous proof only shows that each of the Kashiwara-Malgrange filtrations is strict on the complex, but their intersections may not. 
\end{remark*}
However, in the case of admissible variation of mixed Hodge structures, we have the following.
\begin{coro}\label{C:direct VMHS}
Notations as in Theorem \ref{T:direct image}, instead of assuming that $D^Y$ is smooth, we assume that $\mathcal{H}^if_+(\mathcal{M}[*D^X])$ (resp. $\mathcal{H}^if_+(\mathcal{M}[!D^X])$) is an admissible variation of mixed Hodge structures restricted on $Y\setminus D^Y$. Then we have 
$$
\left(\mathcal{H}^i f_\#  \mathcal{M}\left[*\log D^X\right]\right)^{\vee\vee}= \left(\mathcal{H}^i f_+\mathcal{M}\right)[*\log D^Y].
$$
$$\left(\text{Resp. } \left(\mathcal{H}^if_\#  \mathcal{M}\left[!\log D^X\right]\right)^{\vee\vee}= \left(\mathcal{H}^i f_+\mathcal{M}\right)[!\log D^Y].\right)
$$
\end{coro}
\begin{proof}
Due to the extra assumption, we have that $\mathbf{V}^{D^Y}_\mathbf{0} \left(\mathcal{H}^i f_+\mathcal{M}\right)[*D^Y]$ is locally free, hence reflexive. Now we only need to show the identity out of a closed codimension $2$ subset on $Y$. This is implied by the previous theorem.
\end{proof}
Consider the associated graded complex, by Proposition \ref{P:laumon}, we obtain the following, comparing to \cite[Theorem 2.4]{PS13}.
\begin{coro}\label{C:direct image gr}
Notations as in Theorem \ref{T:direct image} and Proposition \ref{P:laumon}, we have
\begin{align*}
\mathcal{H}^i f_{\widetilde{\#}}\left(\text{Gr}^F  \mathcal{M}\left[*\log D^X\right]\right)&\simeq\text{Gr}^F  \left(\mathcal{H}^i f_+\mathcal{M}\right)[*\log D^Y],\\
\mathcal{H}^i f_{\widetilde{\#}}\left(\text{Gr}^F  \mathcal{M}\left[!\log D^X\right]\right)&\simeq\text{Gr}^F   \left(\mathcal{H}^i f_+\mathcal{M}\right)[!\log D^Y].
\end{align*}
\end{coro}
\section{duality functor}

\begin{proof}[Proof of Theorem \ref{T:dual}]
We only prove the first equation here.

Apply the resolution in Proposition \ref{P:free resolution NC} on $\mathcal{M}[*D]$, locally we have 
$$\left(\mathcal{M}^\bullet, \mathbf{V}_\bullet^D\right)\to \left(\mathcal{M}[*D], \mathbf{V}_\bullet^D\right),$$
where each $\left(\mathcal{M}^i, \mathbf{V}^D_\bullet \right)$ is a direct sum of finite copies of $\left(\widetilde{\mathscr{D}}_X\cdot z^p , \mathbf{V}^D[\mathbf{a}]\right)$, with $-\mathbf1< \mathbf{a}\leq \mathbf0, p\in\mathbb{Z}$.
Now we only need to show that 
\begin{equation}\label{E:dual compare}
\mathbf{R}\mathcal{H}om_{\widetilde{\mathscr{D}}_{(X,D)}}\left(\mathbf{V}^D_{\mathbf{0}}\left(\mathcal{M}^\bullet, \mathbf{V}^D_\bullet\right), \widetilde{\omega}_X[d_X]\otimes \widetilde{\mathscr{D}}_{\left(X,D\right)}\right)=\mathbf{V}^D_{<\mathbf{0}}\mathcal{M}'.
\end{equation}

However, for $-\mathbf{1}< \mathbf{a}\leq \mathbf0$, we have 
$$\mathbf{V}^D_\mathbf{0}\left(\widetilde{\mathscr{D}}_X \cdot z^p, \mathbf{V}^D[\mathbf{a}]\right)=\widetilde{\mathscr{D}}_{\left(X,D\right)}\cdot z^p.$$
Hence
$$
\mathbf{R}\mathcal{H}om_{\widetilde{\mathscr{D}}_{(X,H)}}\left(\mathbf{V}^D_\mathbf{0}\left(\widetilde{\mathscr{D}}_X\cdot z^p , \mathbf{V}^D[\mathbf{a}]\right), \widetilde{\omega}_X[d_X]\otimes \widetilde{\mathscr{D}}_{(X,D)}\right)\simeq\widetilde{\omega}_X[d_X]\otimes  \widetilde{\mathscr{D}}_{(X,D)}\cdot z^{-p}.
$$
On the other hand, to get the Kashiwara-Malgrange filtration $\mathbf{V}^D_\bullet$ on $\mathcal{M}'$, we set 
$$
\mathcal{H}om_{\widetilde{\mathscr{D}}_X} \left(\left(\widetilde{\mathscr{D}}_X\cdot z^p , \mathbf{V}^D[\mathbf{a}]\right), \widetilde{\omega}_X[d_X]\otimes \widetilde{\mathscr{D}}_X \right)=\left(\widetilde{\omega}_X[d_X]\otimes\widetilde{\mathscr{D}}_X\cdot z^{-p}, \mathbf{V}^D[-\mathbf1-\mathbf{a}]\right),
$$
which gives a filtration on the complex 
$$\mathbf{R}\mathcal{H}om_{\widetilde{\mathscr{D}}_X} \left(\left(\mathcal{M}^\bullet, \mathbf{V}^D_\bullet\right), \widetilde{\omega}_X[d_X]\otimes \widetilde{\mathscr{D}}_X \right).$$
This filtration is actually strict and induces the Kashiwara-Malgrange filtration $\mathbf{V}^D_\bullet$ on $\mathcal{M}'$, as shown in the proof of \cite[5.1.13. Lemme.]{Sa88}. Note that Saito only showed the pure Hodge module case there, but the proof still works in the mixed Hodge module case, by induction on the weights.

Hence, for $-\mathbf{1}< \mathbf{a}\leq \mathbf0$, we have
\begin{align*}
&\mathbf{V}^D_{<\mathbf0}\mathbf{R}\mathcal{H}om_{\widetilde{\mathscr{D}}_X}\left(\left(\widetilde{\mathscr{D}}_X\cdot z^p , \mathbf{V}^D[\mathbf{a}]\right), \widetilde{\omega}_X[d_X]\otimes \widetilde{\mathscr{D}}_X \right)\\
\simeq&\mathbf{V}^D_{<\mathbf0}\left(\widetilde{\omega}_X[d_X]\otimes\widetilde{\mathscr{D}}_X\cdot z^{-p}, \mathbf{V}^D[-\mathbf1-\mathbf{a}]\right)\\
=&\widetilde{\omega}_X[d_X]\otimes \widetilde{\mathscr{D}}_{(X,D)}\cdot z^{-p}.
\end{align*}
Now (\ref{E:dual compare}) is clear.

Recall that we use the trivial $\widetilde{\mathscr{D}}$-module structure on $\widetilde{\omega}_X\otimes \widetilde{\mathscr{D}}_X $ to induce the $\widetilde{\mathscr{D}}$-module structure on $\mathcal{M}'$.
To show that $\mathcal{M}'$ also admits a V-compatible multi-indexed Kashiwara-Malgrange filtration with respect to $D$, we only need to check those conditions on $\left(\widetilde{\mathscr{D}}_X, \mathbf{V}^D[-\mathbf1-\mathbf{a}]\right)$, for $-\mathbf{1}< \mathbf{a}\leq \mathbf0$, and it is clear. 
\end{proof}
  
Pass to the associated graded pieces, by using  Proposition \ref{P:dual laumon}, we obtain the following, comparing to \cite[Theorem 2.3]{PS13}.
\begin{coro}\label{C:dual gr}
Notations as in Theorem \ref{T:dual}, we have
\begin{align*}
\text{Gr}^F \mathcal{M}'[*\log D]\simeq&\mathbf{R}\mathcal{H}om_{\mathcal{A}_{\left(X,D\right)}}\left(\text{Gr}^F \mathcal{M}[!\log D], \omega_X[d_X]\otimes_{\mathscr{O}_X}\mathcal{A}_{\left(X,D\right)}\right),\\
\text{Gr}^F  \mathcal{M}'[!\log D]\simeq&\mathbf{R}\mathcal{H}om_{\mathcal{A}_{\left(X,D\right)}}\left(\text{Gr}^F \mathcal{M}[*\log D], \omega_X[d_X]\otimes_{\mathscr{O}_X}\mathcal{A}_{\left(X,D\right)}\right).
\end{align*}
If we consider both sides as graded $\mathscr{O}_{T^*_{\left(X,D \right)}}$-complexes on $T^*_{\left(X,D \right)}$, we get
\begin{align*}
	 \mathcal{G}\left( \mathcal{M}'[*\log D]\right)\simeq& \left(-1\right)^*_{T^*_{\left(X,D\right)}}
\mathbf{R}\mathcal{H}om_{\mathscr{O}_{T^*_{\left(X,D\right)}}}
\left(\mathcal{G}\left( \mathcal{M}[!\log D]\right), p_X^*\omega_X[d_X] \otimes \mathscr{O}_{T^*_{\left(X,D\right)}}\right),\\
\mathcal{G}\left( \mathcal{M}'[!\log D]\right)\simeq& \left(-1\right)^*_{T^*_{\left(X,D\right)}}
\mathbf{R}\mathcal{H}om_{\mathscr{O}_{T^*_{\left(X,D\right)}}}
\left(\mathcal{G}\left(\mathcal{M}[*\log D]\right), p_X^*\omega_X[d_X] \otimes \mathscr{O}_{T^*_{\left(X,D\right)}}\right).
\end{align*}
\end{coro}

\bibliographystyle{alpha}
\bibliography{mybib}
\end{document}